\pdfoutput=1
\documentclass[11pt]{article}

\usepackage{amsmath,amssymb,amsthm,mathtools}
\usepackage[margin=1.15in]{geometry}
\DeclareTextSymbolDefault{\textsection}{OMS}
\usepackage[expansion=false]{microtype}
\usepackage{pgfplots}
\pgfplotsset{compat=1.18}
\usepackage[colorlinks=true,linkcolor=blue,citecolor=blue,urlcolor=blue]{hyperref}
\hypersetup{pdftitle={Asymptotic Root Dynamics under Repeated and Fractional Integration},
  pdfauthor={Patrick Lawless}}

\theoremstyle{plain}
\newtheorem{theorem}{Theorem}[section]
\newtheorem{lemma}[theorem]{Lemma}
\newtheorem{corollary}[theorem]{Corollary}
\newtheorem{proposition}[theorem]{Proposition}
\newtheorem*{cst}{Classical scaling theorem}
\theoremstyle{remark}
\newtheorem{remark}[theorem]{Remark}

\numberwithin{equation}{section}

\newcommand{\LP}{L[P]}
\newcommand{\CP}{C_P}
\newcommand{\DP}{D_P}
\newcommand{\RP}{R_P}
\newcommand{\SP}{S_P}
\DeclareMathOperator{\dist}{dist}
\DeclareMathOperator{\disc}{disc}

\begin{document}

\begin{center}
{\Large\bfseries Asymptotic Root Dynamics under Repeated\\[2pt] and Fractional Integration\par}
\vspace{12pt}
{\large Patrick Lawless\par}
\vspace{4pt}
Independent researcher, Severna Park, MD, USA\\
\texttt{plawless@alumni.nd.edu}\\
ORCID: \href{https://orcid.org/0009-0003-5453-2774}{0009-0003-5453-2774}
\end{center}
\vspace{6pt}

\begin{abstract}
Fix a polynomial $P(x)=\sum_{j=0}^{m}a_jx^j$ of degree $m$, and let
$I_kP$ be its $k$-fold antiderivative with all integration constants zero.
Borwein, Chen, and Dilcher proved that, after $x=ky$, the nonzero zeros of
$I_kP$ converge to the nonzero zeros of $\LP(y)=\sum_j a_j\,j!\,y^j$. That convergence determines the leading linear motion but neither the constant offset nor successive differences. We compute the
next two terms at a simple nonzero zero $\lambda$ of $\LP$. Integer-order and
Riemann--Liouville integration reduce, after scaling, to one rational family
$c_j(h)=\prod_{\ell\le j}(1+\ell h)^{-1}$ sampled at $h=1/k$ and $h=1/\alpha$.
We prove $r_k=k\lambda+\CP(\lambda)+\DP(\lambda)/k+O(k^{-2})$ with explicit
$\CP$ and $\DP$. Hence $\lambda$ is at once the scaled zero, the limit of successive root differences, and the continuous-order velocity, each with an explicit correction; $\CP(\lambda)$ is the offset of an asymptotic line.
For binomials the limiting spacing is a multiple of a gamma-ratio quantity,
the Factorial Fraction,
and the nonzero reduced zeros form an exact regular polygon at every positive
order. For real quadratics the collision order of two branches takes a simple
discriminant form. In every degree the collision discriminant carries a universal
factor, the one familiar from Laguerre polynomials and supported at negative
orders. In the coordinates $(\alpha,x)$ all trajectories lie on one plane curve of
degree $m$ with leading form $\LP$. At a nonzero zero of $\LP$ of multiplicity $q$, the nearby zeros split on the scale
$\sqrt k$ like the zeros of the Hermite polynomial $\mathrm{He}_q$, and for real $P$ the nonzero zeros stay real at every positive order exactly
when $\LP$ is real-rooted.
\end{abstract}



\section{Introduction and relation to previous work}\label{sec:intro}

Repeated integration with zero constants produces a particularly rigid
polynomial family. Since
\begin{equation}\label{eq:monomial}
  I_k[x^j](x)=\frac{j!}{(j+k)!}\,x^{j+k},
\end{equation}
a fixed polynomial acquires gamma-factor weights, and its nonzero zeros move on
the scale $|x|\asymp k$. The known asymptotics identify the slope of that
motion. What is computed here is the next order: a constant term, which turns a
statement about direction into one about position and yields a law for
successive differences, and a further term giving the rate of both.

The leading fixed-polynomial asymptotics were established by Borwein, Chen, and
Dilcher \cite[Prop.~6.4]{BorweinChenDilcher1995}. In their notation, a
normalized $k$-fold integral converges locally uniformly to
\begin{equation}\label{eq:Ldef}
  \LP(z)=\sum_{j=0}^{m}a_j\,j!\,z^j,
\end{equation}
and consequently the nonzero zeros of $I_kP$, divided by $k$, converge as
a multiset to the nonzero zeros of $\LP$. A corollary there places the nonzero
zeros in an annulus $c_-(P)\,k<|z|<c_+(P)\,k$ for large $k$, and a later remark
\cite[\S7, Remark~1]{BorweinChenDilcher1995} notes that the operator-theoretic
results of Prather and Shaw \cite{PratherShaw1983} specialize to this setting
in a considerably weaker form. Later work of Bello-Hern\'andez, Pijeira-Cabrera, and Rivero-Castillo
developed strong asymptotics and limiting zero distributions for related
Borwein--Chen--Dilcher polynomial families \cite{BelloPijeiraRivero2020}.

Fractional-order root motion has been studied from complementary directions.
Caparatta, Pauli, and Saidak study the zero trajectories of fractional
derivatives of polynomials, principally of Riemann--Liouville type, with a
section on the Caputo derivative \cite{CaparattaPauliSaidak2025}. They treat
low-degree cases, including an explicit collision order for quadratics, and
asymptotic phenomena. Wani and Sheikh develop
terminal-dependent geometric results and first-order deformations near the
classical derivative orders for Riemann--Liouville and Caputo operators
\cite{WaniSheikh2026}. Hall, Ho, Jalowy, and Kabluchko study operators of the
form $z^{a}(d/dz)^{b}$ applied $\lceil Nt\rceil$ times to a random polynomial of
degree $N$ \cite{HallHoJalowyKabluchko2026}. Their description of the limiting
root distribution as the push-forward of a transport map is stated for $b>0$;
the case $b<0$, which includes repeated integration, is analyzed in their
\S8 for random polynomials, including regimes in which the zeros concentrate on
circles. That setting is a large-degree,
distributional regime. It belongs to a now substantial body of work on the
flow of roots under \emph{repeated differentiation}. Steinerberger formally
derived a nonlocal transport equation for that flow \cite{Steinerberger2019}.
In a companion note he related the resulting limiting root distribution to
fractional free convolution powers \cite{Steinerberger2023}, in the framework
developed by Shlyakhtenko and Tao \cite{ShlyakhtenkoTao2020}. Kiselev and Tan
analyzed that equation rigorously \cite{KiselevTan2022}: they proved global
regularity in the periodic setting and showed that its solutions track the
zeros of successive derivatives for a class of trigonometric polynomials.
Hoskins and Kabluchko obtained closed formulas in the rotationally invariant
case by a different, numerically corroborated method \cite{HoskinsKabluchko2021}.
Mart\'inez-Finkelshtein and Rakhmanov describe the limiting zero distribution
when the number of derivatives grows in proportion to the degree
\cite{MartinezFinkelshteinRakhmanov2024}. The present
work instead fixes the polynomial and studies individual reduced zeros as the
integration order tends to infinity. We therefore distinguish the present fixed-polynomial, large-integration-order regime from
large-degree limits and from endpoint-order limits.

The contributions proved here are deliberately narrow:
\begin{enumerate}
\item after scaling, integer-order integration and positive-order
Riemann--Liouville integration are exactly the same rational perturbation
problem in a single parameter $h$;
\item at every simple nonzero zero $\lambda$ of $\LP$, the associated root
branch has an explicit expansion through the $1/k$, respectively $1/\alpha$,
term;
\item the expansion yields a limiting successive spacing with an explicit
$O(k^{-2})$ leading correction, and a continuous-order velocity with the
matching $O(\alpha^{-2})$ correction;
\item the constant term is the displacement of a parallel asymptotic line, and
the corresponding polar expansion is explicit;
\item for binomials the nonzero reduced zeros form an exact regular polygon at
every positive order. For quadratics the collision order is expressed through
two discriminants, with a criterion for its positivity; in the derivative
direction it is \cite[Cor.~3.8]{CaparattaPauliSaidak2025}. In arbitrary degree
the collision discriminant carries a universal factor, supported at negative
orders, which is the factor in the classical Laguerre discriminant. The
remaining polynomial has degree at most $m(m-1)/2$. Its positive zeros are
exactly the positive collision orders, and its sign data give a parity
criterion for their existence;
\item the transformed zero $\lambda=0$ is rigid: the origin is pinned at every
positive order and the root count is exact. Restricting the moving-root results
to $\lambda\neq0$ therefore omits no moving branch. The degeneration in which
varying $P$ drives a nonzero transformed zero to the origin is identified with
a merging of discriminant values at $\alpha=-2$;
\item the trajectories form one plane curve of degree $m$ with leading form
$\LP$. On this curve the constant term of item~2 is an asymptote offset and the
collisions of item~5 are discriminant values. The order-zero endpoint is governed
by a harmonic analogue of the transform that controls the endpoint at
infinity;
\item at a zero of $\LP$ of multiplicity $q$, the nearby zeros split on the
scale $\sqrt k$ like the zeros of $\mathrm{He}_q$. In every degree the
nonzero zeros remain real, and simple, for all positive orders exactly when
$\LP$ is real-rooted.
\end{enumerate}

Two settings motivate the refinement. Riemann--Liouville integrals are the
basic operators of fractional calculus, which is used to model memory effects
such as viscoelastic relaxation and anomalous diffusion
\cite{Podlubny1999,KilbasSrivastavaTrujillo2006}. The results here describe how
the zeros of a polynomial profile, and so the sign changes of its fractional
integrals, move as the order varies. The flow of zeros under repeated
differentiation, the inverse operation, is connected with free probability and
random matrices \cite{Steinerberger2019,ShlyakhtenkoTao2020,MarcusSpielmanSrivastava2022}.
The finite free convolution and Laguerre representations of
Remark~\ref{rem:ffree} place repeated integration in the same framework.

The name \emph{Factorial Fraction} is used only for a convenient gamma-ratio
quantity in the binomial case. No novelty is claimed for the underlying gamma
identities.

The paper is organized as follows.
Sections~\ref{sec:classical}--\ref{sec:expansion} fix the operator, identify the
two scaled coefficient families, and prove the local expansion.
Sections~\ref{sec:consequences}--\ref{sec:lowdegree} draw the discrete,
continuous, and geometric consequences and treat binomials and low-degree
checks. Sections~\ref{sec:structural}--\ref{sec:curve} collect structural
properties of $L$, the collision orders, and the plane-curve synthesis.
Sections~\ref{sec:special}--\ref{sec:conclusion} give the special-function
illustrations, the Hermite law at multiple zeros, limitations, scope, and
conclusion.

\section{Repeated integration and the classical scaling limit}
\label{sec:classical}

Let $P(x)=\sum_{j=0}^{m}a_jx^j$ with $a_m\neq0$. Define $I_kP$ by $k$
integrations from $0$, equivalently by the initial-value problem
$y^{(k)}=P$ with $y(0)=y'(0)=\dots=y^{(k-1)}(0)=0$. Termwise integration gives
\begin{equation}\label{eq:Ik}
  I_kP(x)=\sum_{j=0}^{m}a_j\,\frac{j!}{(j+k)!}\,x^{j+k}.
\end{equation}
The weight has a geometric reading that explains the shape of the continuous
interpolation used below. Iterated integration from $0$ is an integral over the
ordered region $\{0\le t_1\le\dots\le t_k\le x\}$, the corner (orthoscheme)
simplex of volume $x^{k}/k!$. The normalization $1/k!$ is therefore the volume of
the unit ordered simplex; for $j=0$, \eqref{eq:monomial} gives $I_k[1](x)=x^{k}/k!$. For real $x>0$,
$x^{\alpha}/\Gamma(\alpha+1)$ in \eqref{eq:H} is the gamma-function
interpolation of that volume formula to noninteger order.

For $x\neq0$, put $x=ky$ and divide away the forced factor $x^{k}$.
Multiplying the reduced equation by $k!$ yields
\begin{equation}\label{eq:Qk}
  Q_k(y)=\sum_{j=0}^{m}a_j\,j!\,c_{j,k}\,y^j,
  \qquad
  c_{j,k}=\frac{k!\,k^{j}}{(k+j)!}=\prod_{\ell=1}^{j}\frac{k}{k+\ell}.
\end{equation}
For each fixed $j$ we have $c_{j,k}\to1$, so $Q_k\to\LP$ locally uniformly.
This limit $\LP$ is exactly the transform used in \cite{BorweinChenDilcher1995}. In
coefficient language it is the inverse of the coefficient map
$a_j\mapsto a_j/j!$ underlying Borel's integral method
\cite[\S\S4.12--4.13]{Hardy1949}; several conventions for Borel transforms are
in use, and the map $a_j\mapsto a_j/j!$ is the one meant here.

\begin{cst}[{\cite[Prop.~6.4]{BorweinChenDilcher1995}}]
Let $s=\operatorname{ord}_0P$. The $m-s$ nonzero zeros of $I_kP$, divided by
$k$, converge as a multiset to the $m-s$ nonzero zeros of $\LP$. In particular, if $\lambda$ is a simple nonzero zero of $\LP$, then for
all sufficiently large $k$ there is a unique scaled zero $y_k$ near $\lambda$,
and $y_k\to\lambda$.
\end{cst}

\section{A single perturbation parameter}\label{sec:parameter}

\subsection{Riemann--Liouville embedding}
For $\alpha>0$ and $x>0$, define the Riemann--Liouville integral
\cite{Podlubny1999,SamkoKilbasMarichev1993,KilbasSrivastavaTrujillo2006}
\begin{equation}\label{eq:RL}
  J^{\alpha}P(x)=\frac{1}{\Gamma(\alpha)}\int_0^{x}(x-t)^{\alpha-1}P(t)\,dt.
\end{equation}
For a monomial the Euler beta integral \cite[\S5.12]{DLMF} gives
\[
  \frac{1}{\Gamma(\alpha)}\int_0^{x}(x-t)^{\alpha-1}t^{j}\,dt
  =\frac{\Gamma(j+1)}{\Gamma(j+\alpha+1)}\,x^{j+\alpha}.
\]
Therefore, by linearity,
\begin{equation}\label{eq:H}
  J^{\alpha}P(x)=x^{\alpha}H(\alpha,x),
  \qquad
  H(\alpha,x)=\sum_{j=0}^{m}a_j\,\frac{j!}{\Gamma(j+\alpha+1)}\,x^{j}.
\end{equation}
At an integer $\alpha=k$ this is precisely $I_kP$. For noninteger $\alpha$ and
complex $x$ the prefactor $x^{\alpha}$ requires a branch choice; throughout,
complex root trajectories mean zeros of the reduced polynomial
$H(\alpha,\cdot)$. On the positive real axis, where $x^{\alpha}>0$, this agrees
with the literal Riemann--Liouville integral.

\subsection{Exact identification of the scaled families}
\label{subsec:identification}
Set $x=\alpha y$ in the reduced equation and multiply by $\Gamma(\alpha+1)$:
\[
  Q(\alpha,y)=\sum_{j=0}^{m}a_j\,j!\,c_j(\alpha)\,y^{j},
  \qquad
  c_j(\alpha)=\frac{\Gamma(\alpha+1)\alpha^{j}}{\Gamma(\alpha+j+1)}
  =\prod_{\ell=1}^{j}\frac{\alpha}{\alpha+\ell}.
\]
Introduce, for a complex parameter $h$,
\begin{equation}\label{eq:cjh}
  c_j(h)=\prod_{\ell=1}^{j}(1+\ell h)^{-1},
  \qquad
  Q_h(y)=\sum_{j=0}^{m}a_j\,j!\,c_j(h)\,y^{j}.
\end{equation}
Since $c_{j,k}=\prod(1+\ell/k)^{-1}$ and $c_j(\alpha)=\prod(1+\ell/\alpha)^{-1}$,
one has exactly $Q_k=Q_h$ at $h=1/k$ and $Q(\alpha,\cdot)=Q_h$ at
$h=1/\alpha$. (The subscript on $Q_k$ is an integration order, that on $Q_h$ a
value of the parameter; the two agree at $h=1/k$, so no confusion arises.) Thus
the discrete and continuous scaled problems are not merely analogous: they are
samples of one rational family. If $y(h)$ denotes a locally selected zero
branch, then
\begin{equation}\label{eq:sampling}
  r_k=k\,y(1/k),\qquad r(\alpha)=\alpha\,y(1/\alpha),
  \qquad\text{hence}\qquad r_k=r(\alpha)\big|_{\alpha=k}
\end{equation}
whenever the same local branch is selected.

\section{The local analytic expansion}\label{sec:expansion}

\subsection{Coefficient expansion}
Define $\sigma_1(j)=\sum_{\ell=1}^{j}\ell=j(j+1)/2$ and
$\sigma_2(j)=\sum_{\ell=1}^{j}\ell^{2}=j(j+1)(2j+1)/6$.

\begin{lemma}\label{lem:coeff}
For $j\ge1$, $c_j(h)$ is rational with poles exactly at $h=-1/\ell$,
$1\le\ell\le j$, the nearest to the origin being at distance $1/j$. Hence for a
polynomial of degree $m$ every coefficient $c_j(h)$ with $0\le j\le m$ is
holomorphic on the common disc $|h|<1/m$, with the convention that there is no
restriction when $m=0$. On that disc
\begin{equation}\label{eq:cexp}
  c_j(h)=1-\sigma_1(j)h
  +\tfrac12\bigl[\sigma_1(j)^{2}+\sigma_2(j)\bigr]h^{2}+O(h^{3}),
\end{equation}
uniformly for $j=0,\dots,m$ on every smaller closed disc.
\end{lemma}

\begin{proof}
The pole statement follows directly from the finite product, and
$|-1/\ell|=1/\ell$ is minimized over $1\le\ell\le j$ at $\ell=j$. For $|h|<1/j$
the logarithm of each factor is given by its convergent series, so
\[
  \log c_j(h)=-\sum_{\ell=1}^{j}\log(1+\ell h)
  =-\sigma_1(j)h+\tfrac12\sigma_2(j)h^{2}+O(h^{3}).
\]
Exponentiating and using
$e^{Ah+Bh^{2}+O(h^{3})}=1+Ah+(A^{2}/2+B)h^{2}+O(h^{3})$ with
$A=-\sigma_1(j)$ and $B=\tfrac12\sigma_2(j)$ gives \eqref{eq:cexp}. Since $j$
ranges over the finite set $\{0,\dots,m\}$, the remainder is uniform on every
closed disc $|h|\le\rho<1/m$.
\end{proof}

The first cases are $c_1=1-h+h^{2}-\dots$, $c_2=1-3h+7h^{2}-\dots$,
$c_3=1-6h+25h^{2}-\dots$, and $c_4=1-10h+65h^{2}-\dots$. Define
\begin{equation}\label{eq:RS}
  \RP(y)=\sum_{j=0}^{m}j(j+1)\,j!\,a_j y^{j},
  \qquad
  \SP(y)=\sum_{j=0}^{m}\bigl[\sigma_1(j)^{2}+\sigma_2(j)\bigr]j!\,a_j y^{j}.
\end{equation}
Then Lemma~\ref{lem:coeff} gives, locally uniformly in $y$,
\begin{equation}\label{eq:Qhexp}
  Q_h(y)=\LP(y)-\tfrac12\RP(y)\,h+\tfrac12\SP(y)\,h^{2}+O(h^{3}).
\end{equation}

\subsection{The branch theorem}

\begin{theorem}[Second-order branch expansion]\label{thm:branch}
Let $\lambda$ be a simple zero of $\LP$, with no assumption that
$\lambda\neq0$. Then there are $\delta>0$ and $h_0\in(0,1/m]$ such that, for $|h|<h_0$, $Q_h$ has exactly one zero $y(h)$ in
the disc $|y-\lambda|<\delta$. The function $y(h)$ is holomorphic there and
admits the convergent expansion
\begin{equation}\label{eq:yexp}
  y(h)=\lambda+\CP(\lambda)h+\DP(\lambda)h^{2}+O(h^{3}),
\end{equation}
where
\begin{equation}\label{eq:CD}
  \CP(\lambda)=\frac{\RP(\lambda)}{2\LP'(\lambda)},
  \qquad
  \DP(\lambda)=\frac{-\LP''(\lambda)\CP(\lambda)^{2}
  +\RP'(\lambda)\CP(\lambda)-\SP(\lambda)}{2\LP'(\lambda)}.
\end{equation}
If $P$ has real coefficients and $\lambda\in\mathbb{R}$, then $y(h)$ is real
for real $h$ sufficiently close to $0$.
\end{theorem}

\begin{proof}
Because the degree is fixed, Lemma~\ref{lem:coeff} implies that
$(h,y)\mapsto Q_h(y)$ is jointly holomorphic on
$\{|h|<1/m\}\times\mathbb{C}$. Choose $\delta>0$ so that the closed disc
$K=\{|y-\lambda|\le\delta\}$ contains no zero of $\LP$ other than $\lambda$.
Since $\LP$ has no zero on $\partial K$, the quantity
$\mu=\min_{\partial K}|\LP|$ is positive. For $|h|$ sufficiently small, local
uniform convergence gives $|Q_h-\LP|<\mu$ on $\partial K$, so Rouch\'e's
theorem \cite[Ch.~4]{Ahlfors1979} implies that $Q_h$ has exactly one zero in
$K$, counted with multiplicity. Since
$\partial_y Q_0(\lambda)=\LP'(\lambda)\neq0$, that zero is simple for all
sufficiently small $|h|$, and it is given by the residue integral
\[
  y(h)=\frac{1}{2\pi i}\oint_{\partial K}
        y\,\frac{\partial_yQ_h(y)}{Q_h(y)}\,dy,
\]
the right-hand side being the sum of the zeros of $Q_h$ in $K$ counted with
multiplicity \cite[Ch.~4]{Ahlfors1979}. Since $Q_h$ has no zero on $\partial K$,
the integrand is holomorphic in $h$ for each $y\in\partial K$ and continuous in
$y$, so $y(h)$ is holomorphic on $|h|<h_0$.

Write $y(h)=\lambda+Ch+Dh^{2}+O(h^{3})$ and substitute into
$0=Q_h(y(h))$, using \eqref{eq:Qhexp}. Since $\LP(\lambda)=0$, the coefficient
of $h$ is $\LP'(\lambda)C-\tfrac12\RP(\lambda)$, which vanishes exactly when
$C=\RP(\lambda)/(2\LP'(\lambda))$. The coefficient of $h^{2}$ is
\[
  \LP'(\lambda)D+\tfrac12\LP''(\lambda)C^{2}
  -\tfrac12\RP'(\lambda)C+\tfrac12\SP(\lambda),
\]
and setting it to zero gives the stated formula for $D$. For real coefficients
and real $h$ the polynomial $Q_h$ has real coefficients, so complex conjugation
preserves its zero set; the unique zero in the conjugation-symmetric disc $K$
must therefore be real.
\end{proof}

\begin{remark}[The expansion continues to all orders]\label{rem:allorders}
Since $y$ is holomorphic at $h=0$ by Theorem~\ref{thm:branch}, its Taylor series
$y(h)=\lambda+\sum_{n\ge1}y_n(\lambda)h^{n}$ converges on a disc. When $a_0\neq0$ and $m\ge2$, the disc has radius at
most $1/m$ (Remark~\ref{rem:negorder}). Therefore
\[
  r_k=k\lambda+y_1(\lambda)+\frac{y_2(\lambda)}{k}+\frac{y_3(\lambda)}{k^{2}}+\cdots,
\]
an asymptotic expansion to every order, with $y_1=\CP$ and $y_2=\DP$. Each
$y_n$ is determined recursively by substituting the series into
$Q_h(y(h))=0$ and using \eqref{eq:cjh}; we stop at $n=2$ because those are the
terms that carry the consequences drawn in Section~\ref{sec:consequences}, not because the
structure terminates.
\end{remark}

\begin{remark}[A closed form for every order]\label{rem:torus}
The scaled polynomial $Q_h$ is $\LP$ with its coefficients $b_j=j!\,a_j$
replaced by $b_jc_j(h)$. So the branch $y(h)$ is the root $\lambda(b)$ of $\LP$,
viewed as a function of the coefficients, evaluated at the point $(b_jc_j(h))_j$. The Euler
operators $\theta_j=b_j\,\partial/\partial b_j$ commute and generate the
scalings $b_j\mapsto b_je^{s_j}$. Write
$\log c_j(h)=\sum_{r\ge1}\kappa_r(j)h^{r}$ with
$\kappa_r(j)=(-1)^{r}\sigma_r(j)/r$ and $\sigma_r(j)=\sum_{\ell\le j}\ell^{r}$.
Then
\[
  y(h)=\exp\Bigl(\sum_{r\ge1}h^{r}\Theta_r\Bigr)\lambda,\qquad
  \Theta_r=\sum_j\kappa_r(j)\,\theta_j,\qquad
  y_n=\frac1{n!}\,B_n\bigl(1!\,\Theta_1,2!\,\Theta_2,\dots,n!\,\Theta_n\bigr)\lambda,
\]
where $B_n$ is the complete Bell polynomial. Implicit differentiation gives
$\theta_j\lambda=-b_j\lambda^{j}/\LP'(\lambda)$, and homogeneity gives
$\sum_jj\,\theta_j\lambda=-\lambda$. Hence
$\CP=\Theta_1\lambda=\bigl(1-\frac12\sum_jj(j-1)\theta_j\bigr)\lambda$, in agreement
with Remark~\ref{rem:compactC}, and $\DP=\bigl(\frac12\Theta_1^{2}+\Theta_2\bigr)\lambda$.
Every $y_n$ is therefore an algebraic function of the coefficients of $P$.
\end{remark}

\begin{remark}[A compact form of $\CP$]\label{rem:compactC}
With $\theta=y\,d/dy$ one has $\theta(\theta+1)y^{j}=j(j+1)y^{j}$, so
$\RP=\theta(\theta+1)\LP=y^{2}\LP''+2y\LP'$. At a simple zero this gives
\[
  \CP(\lambda)=\lambda+\frac{\lambda^{2}\LP''(\lambda)}{2\LP'(\lambda)}.
\]
For the quadratic of Section~\ref{subsec:quadexact}, $\LP=2y^{2}-y-1$ and
$\lambda=1$, so $\CP(1)=1+\frac{4}{2\cdot3}=\frac53$. The identity
also shows that $\RP$ vanishes at every zero of $\LP$ of multiplicity at least
three, which is relevant in Section~\ref{sec:multiple}.
\end{remark}

\begin{remark}[Finite free convolution, Wishart matrices, and a gamma average]\label{rem:ffree}
Rescale $P$ so that $\LP$ is monic, that is, $m!\,a_m=1$, and let $\widehat Q_\alpha=Q_{1/\alpha}/c_m(1/\alpha)$,
which is monic of degree $m$. Since
$c_{m-s}(h)/c_m(h)=\prod_{\ell=m-s+1}^{m}(1+\ell h)$, the elementary symmetric
functions of the zeros satisfy
\[
  e_s(\widehat Q_\alpha)=e_s(\LP)\prod_{\ell=m-s+1}^{m}\Bigl(1+\frac\ell\alpha\Bigr),
  \qquad 0\le s\le m .
\]
This multiplicative rule has three consequences.
\begin{enumerate}
\item It is the finite free multiplicative convolution of Marcus, Spielman, and Srivastava \cite{MarcusSpielmanSrivastava2022}, defined by
$e_s(p\boxtimes_mq)=e_s(p)e_s(q)/\binom ms$:
$\widehat Q_\alpha=\LP\boxtimes_mK_{m,\alpha}$. Here
$K_{m,\alpha}(y)=(-1)^{m}m!\,\alpha^{-m}L_m^{(\alpha)}(\alpha y)$ is a rescaled
Laguerre polynomial, whose zeros are positive.
\item If $A=\operatorname{diag}(\lambda_1,\dots,\lambda_m)$ lists the zeros of
$\LP$ and $W$ is an $m\times m$ complex Wishart matrix with density proportional
to $(\det W)^{\alpha}e^{-\operatorname{tr}W}$, then
$\widehat Q_\alpha(y)=\mathbb{E}\det(yI-AW/\alpha)$. This classical computation rests on the fact that an
$s\times s$ principal minor of $W$ has mean $n(n-1)\cdots(n-s+1)$ with $n=\alpha+m$.
\item If $G_\alpha$ has the gamma law with shape $\alpha+1$ and rate $\alpha$,
then $\mathbb{E} G_\alpha^{j}=1/c_j(1/\alpha)$, so $\mathbb{E}\bigl[Q_{1/\alpha}(yG_\alpha)\bigr]=\LP(y)$
exactly.
\end{enumerate}
By the last item, repeated integration \emph{undoes} a multiplicative gamma
average. $G_\alpha$ concentrates at $1$ with asymptotically Gaussian
fluctuations of size $\alpha^{-1/2}$, and undoing a Gaussian average is a
backward heat flow. This backward heat flow is the mechanism behind the $\sqrt\alpha$ splitting of
multiple zeros in Proposition~\ref{prop:multiple}. For
$\LP=(y-\lambda)^{m}$ item~(1) is exact: the zeros of $J^\alpha P$ are $\lambda$
times the zeros of $L_m^{(\alpha)}$. These zeros lie at
$\alpha+\sqrt\alpha\,\xi_i+(2m+1+\xi_i^{2})/3+O(\alpha^{-1/2})$ with $\mathrm{He}_m(\xi_i)=0$. The leading terms follow from the Hermite limit of
the Laguerre polynomials \cite[\S18.7(iii)]{DLMF}; the constant term was
computed here and checked numerically. When the degree grows with $m/\alpha\to\beta>0$ and the zeros of $\LP$ are
nonnegative with compactly supported limiting distribution $\mu$, the zero
distribution of $\widehat Q_\alpha$ tends to the free multiplicative
convolution of $\mu$ with the limit law of the zeros of $K_{m,\alpha}$
\cite[Thm.~1.4 and Ex.~2.8]{ArizmendiGarzaVargasPerales2023}. That law is a
Marchenko--Pastur law dilated to have mean $1+\beta$, supported on
$[(\sqrt{1+\beta}-\sqrt\beta)^{2},(\sqrt{1+\beta}+\sqrt\beta)^{2}]$. We checked the identity of
item~(1) to $10^{-12}$ for random complex $\LP$ and checked item~(2) by sampling.
The cases $s=1$ and $s=m$ have a geometric reading, valid for every real
$\alpha>-1$. The sum of the $m$ reduced zeros is
$\frac{\alpha+m}{m}\cdot\frac{-a_{m-1}}{a_m}$, so their centroid,
$\frac{\alpha+m}{m^{2}}\cdot\frac{-a_{m-1}}{a_m}$, moves exactly along a ray;
when $a_0\neq0$ these are the nonzero zeros of $J^\alpha P$. When $a_0\neq0$,
the geometric mean of their moduli is
$\bigl|a_0/(m!\,a_m)\bigr|^{1/m}\bigl(\Gamma(\alpha+m+1)/\Gamma(\alpha+1)\bigr)^{1/m}
=\rho\bigl[\alpha+\frac{m+1}2-\frac{m^2-1}{24\alpha}+O(\alpha^{-2})\bigr]$, where
$\rho=|a_0/(m!\,a_m)|^{1/m}$.
For a binomial this is the circumradius of the regular $m$-gon of zeros. When the zeros of $\LP$ are simple and nonzero and lie in strict convex
position, the convex hull of the nonzero zeros of $I_kP$ has area
$k^2A_0+kA_1+A_2+O(1/k)$, with $A_0,A_1,A_2$ given by the shoelace formula
applied to \eqref{eq:yexp}.
\end{remark}

\begin{remark}[Superconvergence]\label{rem:superconv}
Let $\lambda\neq0$ be a simple zero and let $\nu(\lambda)$ be the order of
$y(h)-\lambda$ at $h=0$. Multiplying $Q_h(\lambda)$ by
$\prod_{\ell=1}^{m}(1+\ell h)$ gives
$N_\lambda(h)=\sum_jb_j\lambda^{j}\prod_{\ell=j+1}^{m}(1+\ell h)$, a polynomial of
degree at most $m$ with $[h^m]N_\lambda=m!\,a_0$. Since
$N_\lambda(h)=-\LP'(\lambda)(y(h)-\lambda)(1+O(h))$, we get
$1\le\nu(\lambda)\le m$ when $a_0\neq0$. Then
$r_k-k\lambda=O(k^{1-\nu})$, and for $\nu\ge2$ the spacing error in
Corollary~\ref{cor:discrete} improves to $O(k^{-\nu})$. By
Theorem~\ref{thm:branch} and Remark~\ref{rem:compactC},
$\nu\ge2$ if and only if $\CP(\lambda)=0$, if and only if
$\sum_{\mu\neq\lambda}(\lambda-\mu)^{-1}=-1/\lambda$, the sum running over the other
zeros of $\LP$ counted with multiplicity, a Stieltjes-type
equilibrium condition \cite{Szego1975}. The bound $\nu=m$ forces $N_\lambda=m!\,a_0h^m$, which
fixes $P$ up to scale, and then
$Q_{1/k}(\lambda)=a_0\big/\binom{k+m}{m}$. For $\LP=(y-1)(y-2)$ the branch at
$1$ is $y(h)=1+4h^2-24h^3+\cdots$. When the zeros of $\LP$ are simple and nonzero, at most $m-1$ of them can
have $\CP(\lambda)=0$, since $\sum_i\CP(\lambda_i)/\lambda_i=m(m+1)/2$; this
follows from $\LP''(\lambda_i)/(2\LP'(\lambda_i))=\sum_{j\neq i}(\lambda_i-\lambda_j)^{-1}$. If exactly $m-1$ do, let $a$ be the remaining zero. The
condition $\CP(\lambda)=0$ reads $\lambda\LP''(\lambda)+2\LP'(\lambda)=0$, so the
polynomial $y\LP''+2\LP'$, of degree $m-1$, vanishes at the other $m-1$
zeros, and comparing leading coefficients gives
$y\LP''+2\LP'=m(m+1)\LP/(y-a)$. With $y=az$ this is the hypergeometric
equation $z(1-z)F''+2(1-z)F'+m(m+1)F=0$, so $\LP$ is proportional to
${}_2F_1(-m,m+1;2;z)=(1-z)\,{}_2F_1(1-m,m+2;2;z)$ \cite[\S15.8]{DLMF}, and the
$m-1$ superconvergent zeros are $z=(1-x)/2$ at the zeros $x$ of the Legendre
derivative $P_m'$. For $m=3$ these are $(5\pm\sqrt5)/10$. More generally, we conjecture that
$\sum_i(\nu(\lambda_i)-1)\le m-1$ whenever the zeros of $\LP$ are simple and
nonzero; for $m\le6$ an elimination over all excess patterns supports it,
exactly except for one family of cases checked numerically with explicit
margins, and details will appear separately. Geometrically, the triangle
with vertices $0$, $r_k$, and $r_{k+1}$ has area
$\frac\nu2|\operatorname{Im}(\bar\lambda y_\nu)|\,k^{1-\nu}(1+O(1/k))$ when
$y(h)=\lambda+y_\nu h^\nu+\cdots$ and $\operatorname{Im}(\bar\lambda y_\nu)\neq0$. Indeed, twice that area is
$k(k+1)\bigl|\operatorname{Im}\bigl(\overline{y(1/k)}\,y(1/(k+1))\bigr)\bigr|$, and
$\operatorname{Im}(\overline{y(h_1)}y(h_2))=\operatorname{Im}\bigl(\bar\lambda(y(h_2)-y(h_1))\bigr)+O(h^{2\nu+1})$,
with $h_1=1/k$ and $h_2=1/(k+1)$, because the term $|y_\nu|^2h_1^\nu h_2^\nu$ is
real. So a
zero with $\nu=1$ sweeps asymptotically equal areas
$\frac12|\operatorname{Im}(\bar\lambda\CP(\lambda))|$ about the origin, and a
superconvergent zero does not.
\end{remark}

\section{Consequences for discrete and continuous order}
\label{sec:consequences}

\subsection{Integer order and successive spacing}

\begin{corollary}\label{cor:discrete}
Under the hypotheses of Theorem~\ref{thm:branch}, with $\lambda\neq0$, the corresponding nonzero
zero of $I_kP$ satisfies
\begin{equation}\label{eq:rk}
  r_k=k\lambda+\CP(\lambda)+\frac{\DP(\lambda)}{k}+O(k^{-2}).
\end{equation}
Moreover
\begin{equation}\label{eq:spacing}
  r_k-r_{k-1}=\lambda-\frac{\DP(\lambda)}{k(k-1)}+O(k^{-3}),
\end{equation}
and therefore $r_k-r_{k-1}\to\lambda$.
\end{corollary}

\begin{proof}
Put $h=1/k$ in Theorem~\ref{thm:branch}, admissible for $k>1/h_0$, and use
$r_k=k\,y(1/k)$; this substitution gives \eqref{eq:rk}. Because $y$ is holomorphic at $0$
one may expand one term further,
$y(h)=\lambda+Ch+Dh^{2}+Eh^{3}+O(h^{4})$ for some $E$, whence
$r_k=k\lambda+C+D/k+E/k^{2}+O(k^{-3})$. Subtracting the analogous expansion at
$k-1$ and using $k^{-2}-(k-1)^{-2}=O(k^{-3})$ leaves
$r_k-r_{k-1}=\lambda+D\bigl(\tfrac1k-\tfrac1{k-1}\bigr)+O(k^{-3})$. Since
$\tfrac1k-\tfrac1{k-1}=-\tfrac1{k(k-1)}$, the result follows.
\end{proof}

The additional local expansion is essential. The classical statement
$r_k/k\to\lambda$ alone permits an $o(k)$ error whose successive differences
need not converge, so it does not by itself yield a spacing law.

\subsection{Continuous order and root velocity}

\begin{corollary}\label{cor:continuous}
Under the hypotheses of Theorem~\ref{thm:branch}, for all sufficiently large
$\alpha$ the corresponding reduced Riemann--Liouville root branch satisfies
\begin{equation}\label{eq:ralpha}
  r(\alpha)=\alpha\lambda+\CP(\lambda)+\frac{\DP(\lambda)}{\alpha}
  +O(\alpha^{-2}),
  \qquad
  r'(\alpha)=\lambda-\frac{\DP(\lambda)}{\alpha^{2}}+O(\alpha^{-3}).
\end{equation}
If $P$ is real and $\lambda>0$, then $r(\alpha)>0$ for all sufficiently large
$\alpha$, so the branch is a genuine positive zero of $J^{\alpha}P$.
\end{corollary}

\begin{proof}
Set $h=1/\alpha$ and use $r(\alpha)=\alpha\,y(1/\alpha)$. Since $y$ is
holomorphic at $0$, differentiation is legitimate and
$r'(\alpha)=y(1/\alpha)-\alpha^{-1}y'(1/\alpha)
=\lambda-\DP(\lambda)\alpha^{-2}+O(\alpha^{-3})$. If $\lambda>0$ then
$r(\alpha)/\alpha\to\lambda$, so $r(\alpha)$ is eventually positive; on the
positive axis $x^{\alpha}\neq0$, so a zero of the reduced factor is a zero of
$J^{\alpha}P$.
\end{proof}

\begin{remark}[Spacing as an exact mean velocity]\label{rem:mean}
Since $r_k=r(k)$ on the selected branch by \eqref{eq:sampling}, the fundamental
theorem of calculus gives
\[
  r_k-r_{k-1}=\int_{k-1}^{k}r'(\alpha)\,d\alpha.
\]
The discrete spacing is therefore exactly the mean continuous-order velocity
over one unit interval, and integrating \eqref{eq:ralpha} reproduces the term
$-\DP(\lambda)/(k(k-1))$ of Corollary~\ref{cor:discrete}.
\end{remark}

\begin{remark}[Stability of the branches]\label{rem:nochaos}
Away from multiple zeros, the branches depend stably on the coefficients.
A small change $\delta a$ of the coefficients of $P$ moves a simple zero
$\lambda$ of $\LP$ by $\delta\lambda=O(|\delta a|/|\LP'(\lambda)|)$; this bound
degrades as $\lambda$ approaches another zero. By Theorem~\ref{thm:branch} the
change moves $r_k$ by $k\,\delta\lambda+O(|\delta a|)$, so
$|\delta r_k|\le c\,k\,|\delta a|$ for all large $k$, with $c$ depending on $P$
and $\lambda$. The growth is at most linear, and there is no exponential
sensitivity. Since $J^{a}J^{b}=J^{a+b}$ for $a,b>0$, any sequence of positive orders
$\alpha_1,\alpha_2,\dots$, however irregular, only samples the single branch
$r(\alpha)$ at the cumulative orders $\alpha_1+\dots+\alpha_n$, as long as the
branch stays simple, so that it is continued without ambiguity. Finally, each factor $c_j(1/\alpha)$ with $j\ge1$ increases strictly to $1$ as $\alpha\to\infty$, so the
scaled coefficients approach those of $\LP$ monotonically and
$\sum_j|a_j|\,j!\,(1-c_j(1/\alpha))$ decreases strictly when $P$ is nonconstant.
\end{remark}

\section{Geometry of the asymptotic branch}\label{sec:geometry}

\subsection{Asymptotic line and polar form}
Let $\ell_\lambda=\{\CP(\lambda)+t\lambda:\ t\in\mathbb{R}\}$. The classical
scaling theorem locates the branch only up to direction; the constant term
sharpens this to a statement about position.

\begin{corollary}[Asymptotic line and polar form]\label{cor:geometry}
Under the hypotheses of Theorem~\ref{thm:branch} with $\lambda\neq0$,
$\dist(r(\alpha),\ell_\lambda)=O(\alpha^{-1})$. Moreover the perpendicular
distance from $r(\alpha)$ to the line through the origin in direction $\lambda$
tends to $|\operatorname{Im}(\CP(\lambda)\overline{\lambda})|/|\lambda|$, and
for sufficiently large $\alpha$ the nearest point lies on the ray
$\{t\lambda:\ t>0\}$, so the same limit holds for the distance to that ray.
Writing $r(\alpha)=\rho(\alpha)e^{i\theta(\alpha)}$,
\begin{equation}\label{eq:polar}
  \rho(\alpha)=\alpha|\lambda|
  +|\lambda|\operatorname{Re}\!\Bigl(\frac{\CP(\lambda)}{\lambda}\Bigr)
  +O(\alpha^{-1}),
  \qquad
  \theta(\alpha)=\arg\lambda
  +\frac{\operatorname{Im}\bigl(\CP(\lambda)/\lambda\bigr)}{\alpha}
  +O(\alpha^{-2}),
\end{equation}
and consequently
$\theta'(\alpha)=-\operatorname{Im}(\CP(\lambda)/\lambda)/\alpha^{2}
+O(\alpha^{-3})$.
\end{corollary}

\begin{proof}
The point $\alpha\lambda+\CP(\lambda)$ lies on $\ell_\lambda$, while
Corollary~\ref{cor:continuous} gives
$r(\alpha)-\alpha\lambda-\CP(\lambda)=\DP(\lambda)/\alpha+O(\alpha^{-2})$; this difference proves the first estimate. For $z\in\mathbb{C}$ the distance from the line
$\mathbb{R}\lambda$ is $|\operatorname{Im}(z\overline{\lambda})|/|\lambda|$.
Substituting the expansion of $r(\alpha)$ annihilates the leading term, since
$\operatorname{Im}(\alpha\lambda\overline{\lambda})
=\operatorname{Im}(\alpha|\lambda|^{2})=0$, and leaves the stated limit. The
orthogonal projection parameter is $\alpha+O(1)$, hence positive for large
$\alpha$. For the polar form,
$r(\alpha)/\alpha=\lambda\bigl[1+\CP(\lambda)/(\lambda\alpha)
+O(\alpha^{-2})\bigr]$; using $|1+w|=1+\operatorname{Re}w+O(|w|^{2})$ and
$\arg(1+w)=\operatorname{Im}w+O(|w|^{2})$ gives \eqref{eq:polar}, and
analyticity of the branch justifies differentiating the second expansion.
\end{proof}

For real $P$ and real $\lambda$ the constant $\CP(\lambda)$ is real, the
limiting distance vanishes, and the statement degenerates: the branch lies on
the real axis. The content of Corollary~\ref{cor:geometry} is therefore about
the complex branches, which Section~\ref{subsec:complex} shows can arise
even from real-rooted $P$. If $a_0\neq0$, Vieta's formula for $\LP$ gives
\[
  \prod_{\LP(\lambda)=0}|\lambda|=\frac{|a_0|}{m!\,|a_m|},
\]
counted with multiplicity, so the geometric mean of the asymptotic radial
slopes depends only on the constant and leading coefficients of $P$.

\begin{remark}[Areas are again integrals]\label{rem:area}
Since $J^{\alpha+1}P(x)=\int_0^{x}J^{\alpha}P(t)\,dt$, areas under the curves
studied here are themselves curves of the same kind. In the scaled variable,
$J^{\alpha}P(\alpha y)=(\alpha y)^{\alpha}Q_{1/\alpha}(y)/\Gamma(\alpha+1)$, and
for $Y>0$
\[
  \int_0^{Y}y^{\alpha}Q_{1/\alpha}(y)\,dy
  =\frac{Y^{\alpha+1}}{\alpha+1}\,Q_{1/(\alpha+1)}\Bigl(\frac{\alpha Y}{\alpha+1}\Bigr).
\]
The signed area therefore vanishes exactly when $\alpha Y$ is a nonzero zero of
$J^{\alpha+1}P$, and its zeros move according to Theorem~\ref{thm:branch} at
order $\alpha+1$. 
\end{remark}

\subsection{The exceptional transformed zero \texorpdfstring{$\lambda=0$}{lambda=0}}
Theorem~\ref{thm:branch} is stated without the restriction $\lambda\neq0$, and
at $\lambda=0$ its conclusion is stronger than a generic perturbation
expansion: the branch does not move at all.

\begin{proposition}[The origin is pinned]\label{prop:pinned}
Let $s=\operatorname{ord}_0P$, so that $P(x)=x^{s}\widetilde P(x)$ with
$\widetilde P(0)\neq0$. Then $\operatorname{ord}_0\LP=s$, and for every
$\alpha>0$ the reduced polynomial $H(\alpha,\cdot)$ has a zero of multiplicity
exactly $s$ at the origin. Equivalently $I_kP$ has a zero of multiplicity
$k+s$ there, and exactly $m-s$ nonzero zeros, matching the $m-s$ nonzero zeros
of $\LP$.
\end{proposition}

\begin{proof}
The coefficient map $a_j\mapsto j!\,a_j$ multiplies each coefficient by a
nonzero constant, so it does not change which initial coefficients vanish and
$\operatorname{ord}_0\LP=\operatorname{ord}_0P=s$. In $H(\alpha,x)$ the
coefficient of $x^{j}$ is $a_j\,j!/\Gamma(j+\alpha+1)$, and
$\Gamma(j+\alpha+1)$ is finite and nonzero for $\alpha>0$, so exactly the same
first $s$ coefficients vanish. The integer statement follows from \eqref{eq:Ik},
where each term carries the factor $x^{k}$. Counting degrees, $I_kP$ has degree
$m+k$ and a zero of multiplicity $k+s$ at the origin, leaving $m-s$ nonzero
zeros.
\end{proof}

\begin{corollary}\label{cor:pinned}
If $0$ is a simple zero of $\LP$, that is, if $a_0=0$ and $a_1\neq0$, then the
branch supplied by Theorem~\ref{thm:branch} is $y(h)\equiv0$, and consequently
$\CP(0)=\DP(0)=0$.
\end{corollary}

\begin{proof}
By Proposition~\ref{prop:pinned} with $s=1$, every $Q_h$ has $y=0$ as a simple
zero, so the unique zero near $0$ furnished by Theorem~\ref{thm:branch} is $0$
itself. The vanishing of the coefficients is then immediate, and may also be
read off from \eqref{eq:CD}: $R_P(0)=0$ and $S_P(0)=0$ because the $j=0$ terms
carry the factors $j(j+1)$ and $\sigma_1(j)^{2}+\sigma_2(j)$, while
$\LP'(0)=a_1\neq0$.
\end{proof}

For example, $P(x)=x^{2}-x$ gives $\LP(x)=2x^{2}-x=x(2x-1)$, whereas
\[
  I_kP(x)=x^{k+1}\Bigl(\frac{2x}{(k+2)!}-\frac{1}{(k+1)!}\Bigr)
\]
has the single nonzero zero $r_k=(k+2)/2$, corresponding to $\lambda=1/2$. The
transformed zero $\lambda=0$ carries no moving trajectory at all.

For a fixed polynomial, then, the origin is rigid rather than a second
asymptotic regime, and the hypothesis $\lambda\neq0$ in the corollaries of
Section~\ref{sec:consequences} omits no moving branch. A moving near-zero
regime appears only when $P$ itself varies, which is the subject of
Section~\ref{subsec:degeneration}.
\section{Binomials and the Factorial Fraction}\label{sec:binomials}

Let $A,B>0$ and $a>b>-1$, and consider $f(x)=Ax^{a}-Bx^{b}$, whose positive
zero is $\rho=(B/A)^{1/(a-b)}$. The restriction $b>-1$ is the natural one for the
integral formulation: in \eqref{eq:RL} the factor $t^{b}$ must be integrable at
the lower endpoint, and $t^{b}\in L^{1}_{\mathrm{loc}}$ near $0$ exactly when
$b>-1$. At $b=-1$ the termwise formula fails for an elementary reason as well, namely that
$\int x^{-1}dx=\log x$ is not a power of $x$, so exponents at or below $-1$ fall
outside the integral framework used here. Riemann--Liouville integration of order
$\alpha>0$ acts termwise on $x^{a}$ and $x^{b}$, so the reduced positive zero
is
\begin{equation}\label{eq:binroot}
  r(\alpha)=\rho F(a,b)
  \left[\frac{\Gamma(\alpha+a+1)}{\Gamma(\alpha+b+1)}\right]^{1/(a-b)},
  \qquad
  F(a,b)=\left[\frac{\Gamma(b+1)}{\Gamma(a+1)}\right]^{1/(a-b)}.
\end{equation}
For integer exponents $n_1>n_2\ge0$ this becomes
$F(n_1,n_2)=(n_2!/n_1!)^{1/(n_1-n_2)}$, the reciprocal geometric mean of
$n_2+1,\dots,n_1$. We call $F(n_1,n_2)$ the \emph{Factorial Fraction} of the
exponent pair.

\begin{theorem}[Binomial spacing]\label{thm:binomial}
For fixed $a>b>-1$,
\[
  r(\alpha)=\rho F(a,b)\Bigl(\alpha+\frac{a+b+1}{2}+O(\alpha^{-1})\Bigr).
\]
Consequently, at integer order,
$\lim_{k\to\infty}(r_k-r_{k-1})=\rho F(a,b)$.
\end{theorem}

\begin{proof}
The standard gamma-ratio expansion \cite[\S5.11]{DLMF} is
\[
  \frac{\Gamma(\alpha+a+1)}{\Gamma(\alpha+b+1)}
  =\alpha^{a-b}\Bigl(1+\frac{(a-b)(a+b+1)}{2\alpha}+O(\alpha^{-2})\Bigr).
\]
Raising to the power $1/(a-b)$ and expanding the binomial series gives
$\alpha+(a+b+1)/2+O(\alpha^{-1})$.
\end{proof}

For integer exponents Theorem~\ref{thm:branch} reproduces the same constant
term, which is a nontrivial check on the formula for $\CP$. Indeed, at any
nonzero transformed zero $\lambda$ of
$\LP(y)=An_1!\,y^{n_1}-Bn_2!\,y^{n_2}$ we have
$An_1!\lambda^{n_1}=Bn_2!\lambda^{n_2}=:S$, whence
\[
  \LP'(\lambda)=\frac{S(n_1-n_2)}{\lambda},
  \qquad
  \RP(\lambda)=S\bigl[n_1(n_1+1)-n_2(n_2+1)\bigr],
\]
and since $n_1(n_1+1)-n_2(n_2+1)=(n_1-n_2)(n_1+n_2+1)$ this gives
$\CP(\lambda)=\lambda(n_1+n_2+1)/2$. On the positive branch
$\lambda=\rho F(n_1,n_2)$, so this matches Theorem~\ref{thm:binomial} exactly.

There is also an exact geometric statement, valid at every order rather than
asymptotically. Put $d=n_1-n_2$. Since
$H(\alpha,x)=\frac{An_1!}{\Gamma(n_1+\alpha+1)}x^{n_1}
-\frac{Bn_2!}{\Gamma(n_2+\alpha+1)}x^{n_2}$, the nonzero reduced zeros satisfy
\begin{equation}\label{eq:polygon}
  x^{d}=\frac{Bn_2!}{An_1!}\cdot
  \frac{\Gamma(\alpha+n_1+1)}{\Gamma(\alpha+n_2+1)}.
\end{equation}
The right-hand side is a positive real number for every $\alpha>0$. Hence for
every positive order, integer or fractional, the $d$ nonzero reduced zeros are
exactly the points of a regular root-of-unity configuration centered at the
origin, a regular $d$-gon when $d\ge3$, with fixed
arguments $2\pi q/d$ and a common radius that increases with $\alpha$. The
asymptotic radial speed is $\rho F(n_1,n_2)$. Root-of-unity configurations
of this kind appear already in \cite[\S6, Example~1]{BorweinChenDilcher1995},
where $L$ applied to $\sum_{j=0}^{n}z^{j}/j!$ is computed to be
$(z^{n+1}-1)/(z-1)$, and an exact homothety for a regular-polygon family in the
fractional-derivative direction is obtained in \cite{WaniSheikh2026}.

\section{Low-degree checks}\label{sec:lowdegree}

\subsection{A quadratic with an exact root formula}\label{subsec:quadexact}
For $P(x)=x^{2}-x-1$ we have $\LP(y)=2y^{2}-y-1$, whose positive zero is
$\lambda=1$, and $\RP(y)=12y^{2}-2y$, $\SP(y)=28y^{2}-2y$. Since
$\LP'(1)=3$ and $\LP''=4$,
\[
  \CP(1)=\frac{10}{6}=\frac53,
  \qquad
  \DP(1)=\frac{-4(5/3)^{2}+22(5/3)-26}{6}=-\frac{2}{27}.
\]
Directly, $I_kP(x)=-x^{k}/k!-x^{k+1}/(k+1)!+2x^{k+2}/(k+2)!$, so after division
by $x^{k}$ and multiplication by $(k+2)!$ the nonzero zeros satisfy
$2x^{2}-(k+2)x-(k+1)(k+2)=0$. The positive zero is
\[
  r_k=\frac{(k+2)+\sqrt{9k^{2}+28k+20}}{4}
  =k+\frac53-\frac{2}{27k}+O(k^{-2}),
\]
in exact agreement with Theorem~\ref{thm:branch} and
Corollary~\ref{cor:discrete}.

\subsection{A complex branch arising from a real-rooted polynomial}
\label{subsec:complex}
Let $P(x)=x^{2}-3x+2=(x-1)(x-2)$. Then $\LP(y)=2y^{2}-3y+2$, with zeros
$\lambda_{\pm}=(3\pm i\sqrt7)/4$, both of modulus $1$. Substituting into the
formulas of Theorem~\ref{thm:branch} gives
\[
  \CP(\lambda_{\pm})=\frac32\pm\frac{3\sqrt7}{14}\,i,
  \qquad
  \DP(\lambda_{\pm})=\mp\frac{8\sqrt7}{49}\,i.
\]
Thus a polynomial with two real zeros can generate, at sufficiently large
integration order, a conjugate pair of moving nonzero zeros.
Figure~\ref{fig:branch} shows the trajectory. The two real zeros start at
$1$ and $2$ when $\alpha=0$ and approach one another. They meet at $x_0=12/7$
when $\alpha=\alpha_0=2/7$ and thereafter separate into a conjugate pair. The
upper branch converges not to the ray through the origin in direction
$\lambda_{+}$ but to the parallel line
$\ell_{\lambda_{+}}=\{\CP(\lambda_{+})+t\lambda_{+}\}$. Its limiting
perpendicular distance from the ray is
$|\operatorname{Im}(\CP(\lambda_{+})\overline{\lambda_{+}})|/|\lambda_{+}|
=3\sqrt7/14=0.56695\dots$, exactly as Corollary~\ref{cor:geometry} predicts. Section~\ref{sec:collision} derives the
collision order.

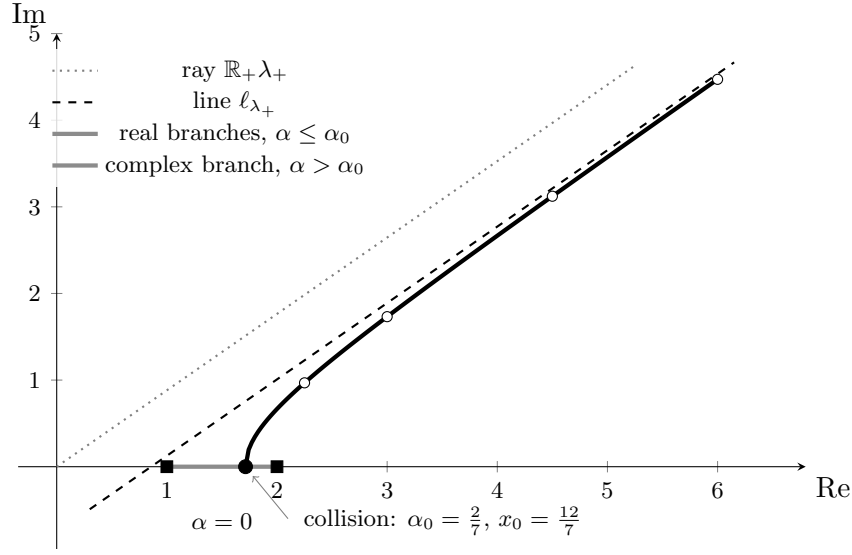
\begin{figure}[!ht]
\centering
\begin{tikzpicture}
\begin{axis}[
  width=12cm, height=8.4cm,
  xmin=-0.35, xmax=6.8, ymin=-0.95, ymax=5.0,
  axis lines=middle,
  xlabel={$\operatorname{Re}$}, ylabel={$\operatorname{Im}$},
  xlabel style={at={(ticklabel* cs:1.0)},anchor=north west},
  ylabel style={at={(ticklabel* cs:1.0)},anchor=south east},
  tick label style={font=\footnotesize},
  legend style={at={(0.03,0.97)},anchor=north west,font=\footnotesize,
                draw=none,fill=white,fill opacity=0.85,text opacity=1},
  clip mode=individual,
]
\addplot[dotted,gray,thick] coordinates {(0.0000,0.0000) (5.2500,4.6301)};
\addlegendentry{ray $\mathbb{R}_{+}\lambda_{+}$}
\addplot[dashed,black,thick] coordinates {(0.3000,-0.4914) (6.1500,4.6679)};
\addlegendentry{line $\ell_{\lambda_{+}}$}
\addplot[black!45,line width=1.6pt] coordinates {(2.0000,0.0000) (2.0000,0.0000) (1.9999,0.0000) (1.9997,0.0000) (1.9995,0.0000) (1.9991,0.0000) (1.9988,0.0000) (1.9983,0.0000) (1.9978,0.0000) (1.9971,0.0000) (1.9964,0.0000) (1.9956,0.0000) (1.9947,0.0000) (1.9937,0.0000) (1.9927,0.0000) (1.9915,0.0000) (1.9902,0.0000) (1.9888,0.0000) (1.9873,0.0000) (1.9857,0.0000) (1.9839,0.0000) (1.9820,0.0000) (1.9800,0.0000) (1.9779,0.0000) (1.9756,0.0000) (1.9731,0.0000) (1.9705,0.0000) (1.9677,0.0000) (1.9647,0.0000) (1.9615,0.0000) (1.9581,0.0000) (1.9544,0.0000) (1.9506,0.0000) (1.9464,0.0000) (1.9420,0.0000) (1.9372,0.0000) (1.9321,0.0000) (1.9267,0.0000) (1.9208,0.0000) (1.9144,0.0000) (1.9075,0.0000) (1.8999,0.0000) (1.8917,0.0000) (1.8825,0.0000) (1.8723,0.0000) (1.8608,0.0000) (1.8476,0.0000) (1.8319,0.0000) (1.8124,0.0000) (1.7855,0.0000) (1.7143,0.0000)};
\addlegendentry{real branches, $\alpha\le\alpha_0$}
\addplot[black!45,line width=1.6pt] coordinates {(1.0000,0.0000) (1.0086,0.0000) (1.0173,0.0000) (1.0260,0.0000) (1.0348,0.0000) (1.0437,0.0000) (1.0527,0.0000) (1.0617,0.0000) (1.0708,0.0000) (1.0800,0.0000) (1.0893,0.0000) (1.0987,0.0000) (1.1081,0.0000) (1.1177,0.0000) (1.1273,0.0000) (1.1371,0.0000) (1.1469,0.0000) (1.1569,0.0000) (1.1670,0.0000) (1.1772,0.0000) (1.1875,0.0000) (1.1980,0.0000) (1.2085,0.0000) (1.2193,0.0000) (1.2301,0.0000) (1.2412,0.0000) (1.2524,0.0000) (1.2638,0.0000) (1.2753,0.0000) (1.2871,0.0000) (1.2991,0.0000) (1.3113,0.0000) (1.3237,0.0000) (1.3364,0.0000) (1.3495,0.0000) (1.3628,0.0000) (1.3764,0.0000) (1.3905,0.0000) (1.4050,0.0000) (1.4199,0.0000) (1.4354,0.0000) (1.4515,0.0000) (1.4683,0.0000) (1.4860,0.0000) (1.5048,0.0000) (1.5249,0.0000) (1.5467,0.0000) (1.5710,0.0000) (1.5991,0.0000) (1.6345,0.0000) (1.7143,0.0000)};
\addplot[black,line width=1.6pt] coordinates {(1.7143,0.0000) (1.7429,0.1968) (1.7714,0.2806) (1.8000,0.3464) (1.8286,0.4032) (1.8571,0.4543) (1.8857,0.5014) (1.9143,0.5457) (1.9429,0.5877) (1.9714,0.6279) (2.0000,0.6667) (2.0286,0.7042) (2.0571,0.7407) (2.0857,0.7762) (2.1143,0.8110) (2.1429,0.8452) (2.1714,0.8787) (2.2000,0.9117) (2.2286,0.9442) (2.2571,0.9762) (2.2857,1.0079) (2.3143,1.0392) (2.3429,1.0702) (2.3714,1.1009) (2.4000,1.1314) (2.4286,1.1616) (2.4571,1.1915) (2.4857,1.2212) (2.5143,1.2508) (2.5429,1.2801) (2.5714,1.3093) (2.6000,1.3383) (2.6286,1.3672) (2.6571,1.3959) (2.6857,1.4245) (2.7143,1.4530) (2.7429,1.4813) (2.7714,1.5095) (2.8000,1.5377) (2.8286,1.5657) (2.8571,1.5936) (2.8857,1.6215) (2.9143,1.6492) (2.9429,1.6769) (2.9714,1.7045) (3.0000,1.7321) (3.0286,1.7595) (3.0571,1.7869) (3.0857,1.8142) (3.1143,1.8415) (3.1429,1.8687) (3.1714,1.8959) (3.2000,1.9230) (3.2286,1.9500) (3.2571,1.9770) (3.2857,2.0040) (3.3143,2.0309) (3.3429,2.0577) (3.3714,2.0846) (3.4000,2.1113) (3.4286,2.1381) (3.4571,2.1648) (3.4857,2.1915) (3.5143,2.2181) (3.5429,2.2447) (3.5714,2.2713) (3.6000,2.2978) (3.6286,2.3243) (3.6571,2.3508) (3.6857,2.3773) (3.7143,2.4037) (3.7429,2.4301) (3.7714,2.4565) (3.8000,2.4828) (3.8286,2.5092) (3.8571,2.5355) (3.8857,2.5617) (3.9143,2.5880) (3.9429,2.6142) (3.9714,2.6405) (4.0000,2.6667) (4.0286,2.6928) (4.0571,2.7190) (4.0857,2.7452) (4.1143,2.7713) (4.1429,2.7974) (4.1714,2.8235) (4.2000,2.8496) (4.2286,2.8756) (4.2571,2.9017) (4.2857,2.9277) (4.3143,2.9537) (4.3429,2.9797) (4.3714,3.0057) (4.4000,3.0317) (4.4286,3.0576) (4.4571,3.0836) (4.4857,3.1095) (4.5143,3.1355) (4.5429,3.1614) (4.5714,3.1873) (4.6000,3.2132) (4.6286,3.2390) (4.6571,3.2649) (4.6857,3.2908) (4.7143,3.3166) (4.7429,3.3425) (4.7714,3.3683) (4.8000,3.3941) (4.8286,3.4199) (4.8571,3.4457) (4.8857,3.4715) (4.9143,3.4973) (4.9429,3.5231) (4.9714,3.5488) (5.0000,3.5746) (5.0286,3.6004) (5.0571,3.6261) (5.0857,3.6518) (5.1143,3.6776) (5.1429,3.7033) (5.1714,3.7290) (5.2000,3.7547) (5.2286,3.7804) (5.2571,3.8061) (5.2857,3.8318) (5.3143,3.8575) (5.3429,3.8831) (5.3714,3.9088) (5.4000,3.9345) (5.4286,3.9601) (5.4571,3.9858) (5.4857,4.0114) (5.5143,4.0371) (5.5429,4.0627) (5.5714,4.0883) (5.6000,4.1139) (5.6286,4.1395) (5.6571,4.1652) (5.6857,4.1908) (5.7143,4.2164) (5.7429,4.2420) (5.7714,4.2676) (5.8000,4.2931) (5.8286,4.3187) (5.8571,4.3443) (5.8857,4.3699) (5.9143,4.3955) (5.9429,4.4210) (5.9714,4.4466) (6.0000,4.4721)};
\addlegendentry{complex branch, $\alpha>\alpha_0$}
\addplot[only marks,mark=*,mark size=1.9pt,black,mark options={fill=white}] coordinates {(2.2500,0.9682) (3.0000,1.7321) (4.5000,3.1225) (6.0000,4.4721)};
\addplot[only marks,mark=square*,mark size=2.2pt,black] coordinates {(1.0000,0.0000) (2.0000,0.0000)};
\addplot[only marks,mark=*,mark size=2.6pt,black] coordinates {(1.7143,0.0000)};
\node[anchor=north,font=\footnotesize,align=center] at (axis cs:1.5,-0.42)
  {$\alpha=0$};
\node[anchor=west,font=\footnotesize] at (axis cs:2.15,-0.62)
  {collision: $\alpha_0=\tfrac27$, $x_0=\tfrac{12}{7}$};
\draw[->,gray,shorten >=2pt] (axis cs:2.10,-0.60) -- (axis cs:1.75,-0.06);
\end{axis}
\end{tikzpicture}
\caption{Root trajectory for $P(x)=(x-1)(x-2)$ under $J^{\alpha}$. The two real
zeros collide at $x_0=12/7$ when $\alpha=\alpha_0=2/7$ and leave the real axis
as a conjugate pair; only the upper branch is drawn, the lower being its mirror
image. The branch is asymptotic to the line
$\ell_{\lambda_{+}}=\{\CP(\lambda_{+})+t\lambda_{+}\}$ (dashed), not to the ray
$\mathbb{R}_{+}\lambda_{+}$ through the origin (dotted); the two are parallel
and separated by $3\sqrt7/14\approx0.567$. Gray: the two real branches for
$0\le\alpha\le\alpha_0$; black: the complex branch for $\alpha>\alpha_0$. Filled
squares mark the zeros $1$ and $2$ of $P$ at $\alpha=0$, the filled circle marks
the collision, and open circles mark the complex branch at
$\alpha=1,2,4,6$.}
\label{fig:branch}
\end{figure}

\subsection{Higher degree}
For $P(x)=ax^{3}+bx^{2}+cx+d$ the transform is $\LP(x)=6ax^{3}+2bx^{2}+cx+d$,
and a quartic transforms to $24ax^{4}+6bx^{3}+2cx^{2}+dx+e$. No radical expression for
$\lambda$ is needed at any degree: it is characterized implicitly by
$\LP(\lambda)=0$, and $\CP(\lambda)$ and $\DP(\lambda)$ are then rational
functions of $\lambda$ with coefficients rational in those of $P$.

\subsection{Exact identities in every degree}\label{subsec:sumrules}
Vieta's formulas give the collective behavior of the nonzero zeros exactly.
Let $a_0a_m\neq0$, let $r_1,\dots,r_m$ be the nonzero zeros of $I_kP$, and let
$e_s$ be their elementary symmetric functions. Since the reduced polynomial has
coefficients $a_j\,j!/(j+k)!$,
\begin{equation}\label{eq:pascal}
  e_s=(-1)^{s}\,\frac{a_{m-s}}{a_m}\,\binom{k+m}{s}\Big/\binom{m}{s},
  \qquad 1\le s\le m,
\end{equation}
a polynomial of degree $s$ in $k$, and the same holds for $J^{\alpha}P$ with the
binomial coefficient read as a polynomial in $\alpha$. For $s=m$ this is the
product $\prod r_i=(-1)^{m}(a_0/a_m)\binom{k+m}{m}$, which for a binomial
$a_0+a_mx^m$ fixes the common radius of the regular polygon of
Section~\ref{sec:binomials}; the factors
$(k+m)(k+m-1)\cdots(k+m-s+1)$ vanish at the negative integers that carry the
universal factor of Proposition~\ref{prop:spurious}.

Comparing \eqref{eq:pascal} for $s=1$ with Theorem~\ref{thm:branch} summed over
the branches gives exact sum rules. If the zeros $\lambda_1,\dots,\lambda_m$ of
$\LP$ are simple, then $e_1=-(a_{m-1}/a_m)(1+k/m)$ has no term in $1/k$, and
\begin{equation}\label{eq:sumrules}
  \sum_i\lambda_i=-\frac{a_{m-1}}{m\,a_m},\qquad
  \sum_i\CP(\lambda_i)=-\frac{a_{m-1}}{a_m},\qquad
  \sum_i\DP(\lambda_i)=0 .
\end{equation}
We checked \eqref{eq:sumrules} on random complex polynomials of degrees three
and four, where the last sum vanishes to within $10^{-15}$. It is a convenient
test of \eqref{eq:CD}, and the higher $e_s$ give further identities of the same
kind.

\section{Structural properties of the transform \texorpdfstring{$L$}{L}}
\label{sec:structural}

The transform has the Laplace representation
\begin{equation}\label{eq:laplace}
  \LP(x)=\int_0^{\infty}e^{-t}P(xt)\,dt,
\end{equation}
valid for every complex $x$, because $\int_0^{\infty}e^{-t}t^{j}\,dt=j!$ and
the sum defining $P$ is finite, so the integral converges absolutely.

Equivalently, since the Laplace transform of $u^{j}$ is $j!/s^{j+1}$ for
$\operatorname{Re}s>0$, one has
$\mathcal{L}[P](s)=\sum_j a_j\,j!\,s^{-(j+1)}$ and hence
\begin{equation}\label{eq:laplacezero}
  s\,\mathcal{L}[P](s)=\LP(1/s),
\end{equation}
so the nonzero zeros of $\LP$ are exactly the reciprocals of the zeros of the
rational function $\sum_ja_j\,j!\,s^{-(j+1)}$. This function is the continuation of
$\mathcal{L}[P]$ from $\operatorname{Re}s>0$ to $s\neq0$. The asymptotic root scale $\lambda$ of
Theorem~\ref{thm:branch} is therefore a reciprocal Laplace zero, and
Proposition~\ref{prop:noncreation} below is the statement that a nonnegative
kernel has no Laplace zero on the positive axis.

\begin{remark}[$L$ as an expectation]\label{rem:expectation}
Since a standard exponential variable $T$ has density $e^{-t}$ on $(0,\infty)$
and moments $\mathbb{E}[T^{j}]=j!$, the representation \eqref{eq:laplace} says
exactly that
\[
  \LP(x)=\mathbb{E}\bigl[P(xT)\bigr],\qquad T\sim\mathrm{Exponential}(1).
\]
The factorial weights are the moments of $T$, and $L$ averages $P$ over an
exponentially distributed rescaling of its argument. The same reading applies at
finite order and explains the limit. If $U_\alpha$ has the $\mathrm{Beta}(1,\alpha)$
density $\alpha(1-u)^{\alpha-1}$ on $[0,1]$, then
$\mathbb{E}[U_\alpha^{\,j}]=j!\,\Gamma(\alpha+1)/\Gamma(j+\alpha+1)$, so the scaled family of Section~\ref{subsec:identification} is
\begin{equation}\label{eq:betarep}
  Q(\alpha,y)=\mathbb{E}\bigl[P(y\,\alpha U_\alpha)\bigr].
\end{equation}
Since $\alpha U_\alpha\Rightarrow T$ as $\alpha\to\infty$, with
$\mathbb{E}[(\alpha U_\alpha)^{j}]\to j!$, the classical scaling limit is the
convergence of these rescaled Beta variables to a standard exponential, and
\eqref{eq:laplace} is its limiting form. This limiting form makes the next proposition immediate, and it suggests replacing $T$ by an arbitrary positive
random variable; see Section~\ref{sec:limitations}.
\end{remark}

\begin{proposition}[Positive-root noncreation]\label{prop:noncreation}
If $P$ is real and has no zero on $(0,\infty)$, then $\LP$ has no zero on
$(0,\infty)$ and has the same constant sign there as $P$.
\end{proposition}

\begin{proof}
A continuous real polynomial without a positive zero has a fixed sign
$\varepsilon\in\{\pm1\}$ on $(0,\infty)$. For $x>0$ the integrand
$e^{-t}P(xt)$ then has sign $\varepsilon$ for every $t>0$ and is continuous and
not identically zero, so $\varepsilon\LP(x)>0$.
\end{proof}

Since $j!>0$, the sequences $(a_j)$ and $(j!\,a_j)$ have the same sign pattern,
so $P$ and $\LP$ have the same number of Descartes sign variations and hence
the same Descartes bound and the same parity for the number of positive zeros
\cite[\S10.1]{RahmanSchmeisser2002}. A sharper statement holds, but it requires
care with terminology.

Borwein, Chen, and Dilcher note explicitly that $L$ is a zero-diminishing
linear transformation: for a real polynomial, $\LP$ cannot have more real
zeros, counted with multiplicity, than $P$
\cite[\S6]{BorweinChenDilcher1995}. The property itself is the subject of
\cite{CravenCsordas1980}. We use that result only as a cited structural fact;
it is not needed in the proof of Theorem~\ref{thm:branch}.

The statement should not be confused with preservation of real-rootedness.
Indeed $P(x)=(1+x)^{2}$ has two real zeros, whereas $\LP(x)=1+2x+2x^{2}$ has
discriminant $-4$, so $L$ can strictly reduce the number of real zeros. The
Hermite polynomials show how severe that reduction can be.

\begin{proposition}\label{prop:hermite}
Let $H_n$ denote the (physicists') Hermite polynomials. Then
\begin{equation}\label{eq:hermitetrunc}
  L[H_n](x)=n!\,(2x)^{n}\,e_M\!\Bigl(-\frac{1}{4x^{2}}\Bigr),
  \qquad M=\lfloor n/2\rfloor,
\end{equation}
where $e_M(z)=\sum_{m=0}^{M}z^{m}/m!$ is the truncated exponential. Consequently
$L[H_n]$ has exactly $n\bmod4$ real zeros, counted with multiplicity, although
$H_n$ has $n$.
\end{proposition}

\begin{proof}
In $H_n(x)=\sum_{m}(-1)^{m}\frac{n!}{m!\,(n-2m)!}(2x)^{n-2m}$ the coefficient of
$x^{j}$ is nonzero only for $j\equiv n\pmod 2$, and for $j=n-2m$ equals
$(-1)^{m}n!\,2^{j}/(m!\,j!)$. Multiplying it by $j!$, as $L$ does, cancels the
$j!$ and leaves
$L[H_n](x)=\sum_{m=0}^{M}(-1)^{m}n!\,(2x)^{n-2m}/m!$, which is
\eqref{eq:hermitetrunc}.

The lowest power of $x$ occurring is $x^{n-2M}$, so the origin is a zero of
multiplicity $n\bmod2$. For real $x\neq0$ the quantity $z=-1/(4x^{2})$ runs over
$(-\infty,0)$, and each negative zero $z_0$ of $e_M$ contributes the two real
zeros $x=\pm1/(2\sqrt{-z_0})$. It therefore suffices to know that $e_M$ has no
real zero for $M$ even and exactly one, negative and simple, for $M$ odd. This property of $e_M$ follows by induction from $e_M'=e_{M-1}$. Given $e_{2k}>0$, the function
$e_{2k+1}$ is strictly increasing and of odd degree. It therefore has a unique
real zero $z_0$, which is negative because $e_{2k+1}(0)=1$. Then $e_{2k+2}$
attains its minimum at $z_0$, where its value is
$e_{2k+1}(z_0)+z_0^{2k+2}/(2k+2)!=z_0^{2k+2}/(2k+2)!>0$, so $e_{2k+2}>0$. Since
$M=\lfloor n/2\rfloor$ is even exactly when $n\equiv0,1\pmod4$, the count is
$(n\bmod2)+2\,[\,n\equiv2,3\pmod 4\,]=n\bmod4$.
\end{proof}

By \eqref{eq:hermitetrunc} the transform carries
a polynomial all of whose zeros are real to one whose zeros are those of a
truncated exponential, and by Szeg\H o's theorem \cite{Szego1924} the limit set of the
scaled zeros of $e_M(Mz)$ is the Szeg\H o curve
$\{|z\,e^{1-z}|=1,\ |z|\le1\}$. With $z=-1/(4x^{2})$ and $z=Mw$, the rescaled nonzero zeros $\sqrt M\,x$ of $L[H_n]$ therefore accumulate on the closed curve
$\{\pm\frac12(-w)^{-1/2}:w\text{ on the Szeg\H o curve}\}$. So $L$ moves the zeros
of $H_n$ off the real axis, leaving at most three behind. In particular $L$ is not a
P\'olya--Schur multiplier sequence, as the P\'olya--Schur criterion also shows
directly \cite{PolyaSchur1914,RahmanSchmeisser2002}: the function associated with
$\gamma_j=j!$ is
\[
  \sum_{j\ge0}\gamma_j\frac{x^{j}}{j!}=\sum_{j\ge0}x^{j}=\frac{1}{1-x},
\]
which is not entire. By contrast the inverse coefficient
map $a_j\mapsto a_j/j!$ is a classical multiplier sequence. Indeed,
$\sum_{j\ge0}x^{j}/(j!)^{2}=I_0(2\sqrt{x})$ is entire of order $1/2$. The
substitution $x=-(z/2)^{2}$ turns it into $J_0(z)$, so its zeros lie only at
$x=-j_{0,n}^{2}/4$, with $j_{0,n}$ the positive zeros of $J_0$
\cite[\S10.21]{DLMF}. It therefore lies in the Laguerre--P\'olya class.

The relations between multiplier sequences, zero-diminishing sequences, and
complex zero decreasing sequences require care, and we do not attempt to state
them here. The P\'olya--Schur characterization of multiplier sequences is
classical; broader classification questions, including those for complex zero
decreasing sequences, remain; see Craven and Csordas
\cite{CravenCsordas1980,CravenCsordas1977,CravenCsordas1995,CravenCsordas2004}. None of those classification
questions is needed below.

\section{Collision order and the root-velocity equation}\label{sec:collision}

Theorem~\ref{thm:branch} is a local statement in $h$ and presupposes that the
selected branch stays simple. Section~\ref{subsec:complex} shows that real
branches can become complex, and it is natural to ask at which order this
happens. Throughout this section put
\begin{equation}\label{eq:Galpha}
  G_\alpha(x):=\Gamma(\alpha+m+1)H(\alpha,x)
  =\sum_{j=0}^{m}b_j(\alpha)x^{j},
  \qquad
  b_j(\alpha)=a_j\,j!\prod_{i=j+1}^{m}(\alpha+i),
\end{equation}
which is a polynomial in $x$ of degree $m$ for every $\alpha$. It has the same
zeros as $H(\alpha,\cdot)$ except at $\alpha=-m-1,-m-2,\dots$, where
$H(\alpha,\cdot)$ vanishes identically and $G_\alpha$ is defined by the product
formula. Write $\Delta_\alpha=\disc_x G_\alpha$. We call
$\alpha$ a \emph{collision order} when $G_\alpha$ has a multiple zero, that is,
when $\Delta_\alpha=0$. The definition is algebraic. For generic $P$ it is the
order at which two previously distinct branches meet, but a repeated factor or
other degeneracy can make $\Delta_\alpha$ vanish without that two-sheet picture,
and no claim about the local sheet structure is intended by the word. Note that $\deg_\alpha b_j\le m-j$, with equality when $a_j\neq0$, and that the
coefficient of $\alpha^{m-j}$ in $b_j$ is $a_j\,j!$, the corresponding
coefficient of $\LP$.

\subsection{Exact collision order for quadratics}\label{subsec:quadcoll}
Let $P(x)=a_0+a_1x+a_2x^{2}$ with $a_0a_2\neq0$. Here
$\Gamma(\alpha+3)/\Gamma(\alpha+1)=(\alpha+1)(\alpha+2)$ and
$\Gamma(\alpha+3)/\Gamma(\alpha+2)=\alpha+2$, so
\begin{equation}\label{eq:eqquad}
  G_\alpha(x)=2a_2x^{2}+a_1(\alpha+2)x+a_0(\alpha+1)(\alpha+2).
\end{equation}
The discriminant factors, and the factorization is the whole content of the
quadratic case.

\begin{lemma}\label{lem:discfactor}
For $P$ of degree two,
\begin{equation}\label{eq:discfactor}
  \Delta_\alpha=(\alpha+2)\bigl[\alpha\disc\LP+2\disc P\bigr],
  \qquad \disc P=a_1^{2}-4a_0a_2,\quad \disc\LP=a_1^{2}-8a_0a_2.
\end{equation}
In particular $\Delta_0=4\disc P$ and the leading coefficient of
$\Delta_\alpha$ in $\alpha$ is $\disc\LP$.
\end{lemma}

\begin{proof}
From \eqref{eq:eqquad},
$\Delta_\alpha=a_1^{2}(\alpha+2)^{2}-8a_0a_2(\alpha+1)(\alpha+2)
=(\alpha+2)\bigl[a_1^{2}(\alpha+2)-8a_0a_2(\alpha+1)\bigr]$, and the bracket
equals $\alpha(a_1^{2}-8a_0a_2)+2(a_1^{2}-4a_0a_2)$. Evaluating at $\alpha=0$
gives $2\cdot2\disc P$.
\end{proof}

Thus, after removal of the factor $\alpha+2$ (positive for $\alpha>-1$), the
discriminant is an affine function of $\alpha$. Its intercept is $2\disc P$ and
its slope is $\disc\LP$. It therefore interpolates linearly between the
discriminant of $P$, recovered at $\alpha=0$, and that of $\LP$, which governs
the behavior as $\alpha\to\infty$. The collision is simply where that
line crosses zero.

The next lemma shows that only one of the two conceivable sign patterns can
occur.

\begin{lemma}\label{lem:signs}
For real $P$ of degree two, $\disc P<0$ implies $\disc\LP<\disc P<0$. Consequently
$\disc P$ and $\disc\LP$ have opposite signs if and only if
$\disc P>0>\disc\LP$.
\end{lemma}

\begin{proof}
If $\disc P<0$ then $a_1^{2}<4a_0a_2$, so $a_0a_2>0$ and hence
$8a_0a_2>4a_0a_2$. Therefore
$\disc\LP=a_1^{2}-8a_0a_2<a_1^{2}-4a_0a_2=\disc P<0$.
\end{proof}

\begin{theorem}[Quadratic collision order]\label{thm:collision}
Let $P(x)=a_0+a_1x+a_2x^{2}$ with $a_0a_2\neq0$ and $\disc\LP\neq0$. Then, apart from the fixed value $\alpha=-2$ at which the factor $\alpha+2$ of
\eqref{eq:discfactor} vanishes, $G_\alpha$ has a repeated zero for exactly one
value of $\alpha$, namely
\begin{equation}\label{eq:alpha0}
  \alpha_0=-\,\frac{2\disc P}{\disc\LP},
\end{equation}
and the repeated zero is
\begin{equation}\label{eq:x0}
  x_0=\frac{2a_0a_1}{\disc\LP}.
\end{equation}
If $P$ is real, the collision occurs at positive integration order if and only
if $\disc P>0>\disc\LP$, that is, if and only if $P$ has two distinct real
zeros while $\LP$ has none. In terms of $t=a_1^{2}/(a_0a_2)$, this happens
exactly when $4<t<8$, and then
\begin{equation}\label{eq:alpha0t}
  \alpha_0=\frac{2(t-4)}{8-t},
\end{equation}
an increasing bijection from $(4,8)$ onto $(0,\infty)$.
\end{theorem}

\begin{proof}
By Lemma~\ref{lem:discfactor} and $\alpha+2\neq0$ for $\alpha>-1$, collisions
are the zeros of $\alpha\disc\LP+2\disc P$, which is a genuinely affine
function since $\disc\LP\neq0$; its unique zero is \eqref{eq:alpha0}. At a
repeated zero, $x_0=-b_1(\alpha_0)/(2b_2)=-a_1(\alpha_0+2)/(4a_2)$. Using
\eqref{eq:alpha0},
$\alpha_0+2=(-2\disc P+2\disc\LP)/\disc\LP=-8a_0a_2/\disc\LP$, hence
$x_0=-a_1\cdot(-8a_0a_2)/(4a_2\disc\LP)=2a_0a_1/\disc\LP$, which is
\eqref{eq:x0}. Positivity of $\alpha_0$ requires $\disc P$ and $\disc\LP$ to
have opposite signs, and Lemma~\ref{lem:signs} excludes the pattern
$\disc P<0<\disc\LP$. Finally, dividing numerator and denominator of
\eqref{eq:alpha0} by $a_0a_2$ gives
$\alpha_0=-2(t-4)/(t-8)=2(t-4)/(8-t)$, and $\disc P>0>\disc\LP$ is exactly
$4a_0a_2<a_1^{2}<8a_0a_2$ with $a_0a_2>0$, that is, $4<t<8$. On that interval
$d\alpha_0/dt=8/(8-t)^{2}>0$, with limits $0$ and $+\infty$ at the endpoints.
\end{proof}

\begin{remark}[Relation to the derivative direction]\label{rem:cps}
Under the change of order $q=-\alpha$, the value \eqref{eq:alpha0} is the
collision order found by Caparatta, Pauli, and Saidak
\cite[Cor.~3.8]{CaparattaPauliSaidak2025} for fractional derivatives of a
monic quadratic $x^{2}+c_1x+c_0$. Their value is
$q_0=1-c_1^{2}/(8c_0-c_1^{2})$, and with $c_1=a_1/a_2$ and $c_0=a_0/a_2$ it
equals $2\disc P/\disc\LP$. What
Theorem~\ref{thm:collision} adds is the expression through the two
discriminants, the criterion for a positive integration order, and the
place of the formula in the general collision theory of
Section~\ref{subsec:gencoll}.
\end{remark}

\begin{corollary}\label{cor:endpoint}
Let $P$ be a real quadratic with $a_0a_2\neq0$ and real zeros. Then the zeros of
$G_\alpha$ are real for every $\alpha>0$ if and only if $\disc\LP\ge0$, that is,
if and only if $\LP$ is itself real-rooted.
\end{corollary}

\begin{proof}
We argue directly from \eqref{eq:discfactor}, which requires no hypothesis on
either discriminant. Real-rootedness of $P$ is $\disc P\ge0$, and for
$\alpha>0$ the factor $\alpha+2$ is positive, so the sign of $\Delta_\alpha$ is
that of $\alpha\disc\LP+2\disc P$. If $\disc\LP\ge0$ this is nonnegative for
every $\alpha>0$, so the zeros stay real. If $\disc\LP<0$ it is negative for
$\alpha>-2\disc P/\disc\LP$, so the zeros are nonreal from some order onward.
The argument covers $\disc P=0$, where $P$ has a repeated zero and the
threshold is $\alpha=0$.
\end{proof}

In other words the endpoint of the flow controls the whole path: in degree two
the zeros survive every order exactly when they survive the limit. Explicitly
the condition is $a_1^{2}\ge8a_0a_2$, to be compared with $a_1^{2}\ge4a_0a_2$
for $P$ itself.

The endpoint criterion holds in every degree.

\begin{proposition}\label{prop:endpointall}
Let $P$ be a real polynomial of degree $m$. The zeros of $J^\alpha P$ other than
$0$ are real for every $\alpha>0$ if and only if $\LP$ is real-rooted, and
already if they are real for a sequence $\alpha_n\to\infty$. When $\LP$ is
real-rooted, these zeros are moreover simple for every $\alpha>0$, even if
$\LP$ has multiple zeros; if also $a_0\neq0$, then the polynomial $\Phi_P$ of
Proposition~\ref{prop:spurious} satisfies $\Phi_P(\alpha)>0$ for every
$\alpha>0$.
\end{proposition}

\begin{proof}
The nonzero zeros of $J^\alpha P$ are $\alpha$ times the nonzero zeros of $Q_{1/\alpha}=\sum_jb_jc_j(1/\alpha)y^{j}$, $b_j=j!a_j$. Fix $\alpha>0$, and let $(x)_j=x(x+1)\cdots(x+j-1)$ denote the rising
factorial, with $(x)_0=1$. The sequence $c_j(1/\alpha)=\alpha^{j}/(\alpha+1)_j$
is a multiplier sequence in the
sense of P\'olya and Schur \cite{PolyaSchur1914,CravenCsordas2004}, because its generating function
$\sum_jc_j(1/\alpha)z^{j}/j!={}_0F_1(;\alpha+1;\alpha z)$ is, up to a factor, a
Bessel function of order $\alpha>-1$. That function is entire of order $\frac12$
with only negative zeros, hence in the Laguerre--P\'olya class. Equivalently, its
Jensen polynomials are the rescaled Laguerre polynomials
$n!\,L_n^{(\alpha)}(-\alpha z)/(\alpha+1)_n$, whose zeros are negative. So
$\LP$ real-rooted implies $Q_{1/\alpha}$ real-rooted. Conversely the
coefficients of $Q_{1/\alpha}$ tend to those of $\LP$, with nonzero leading
coefficient, so by Hurwitz's theorem a limit of real-rooted polynomials of degree
$m$ is real-rooted.

For the simplicity statement, let $\theta=y\,d/dy$, and for $b>0$ define $T_b$ by
$T_b[y^{j}]=y^{j}/(b)_j$. Then $Q_{1/\alpha}(y)=T_{\alpha+1}[\LP](\alpha y)$. Since
$(b)_j(b+j)_N=(b)_N(b+N)_j$, for every $N\ge m$
\[
  T_b=\prod_{l=0}^{N-1}\Bigl(1+\frac{\theta}{b+l}\Bigr)\,T_{b+N}.
\]
By the first part, $T_{b+N}$ preserves real-rootedness. Each factor sends $f$ to
a multiple of $cf+yf'$ with $c>0$. If $f$ is real-rooted of degree $m$, then
so is $cf+yf'$. Indeed, at a nonzero zero of multiplicity $d$ it has a zero of
multiplicity exactly $d-1$, and it keeps the order of vanishing at $0$. By
Rolle's theorem applied to $|y|^{c}f$ on each half-line, it also has a zero
between $0$ and the nearest nonzero zero on each side and between consecutive
distinct nonzero zeros; this completes the count. These further zeros are the
zeros of $c/y+f'/f$, whose derivative $-c/y^{2}-\sum_r(y-r)^{-2}$ is negative,
so they are simple. Each factor therefore lowers the multiplicity of every
multiple nonzero zero by one and creates only simple ones, and after $N\ge m$
factors every nonzero zero is simple. If $a_0\neq0$, all $m$ zeros of the real
polynomial $G_\alpha$ are then real, nonzero and simple, so
$\Delta_\alpha>0$; the universal factor of \eqref{eq:spurious} is positive for
$\alpha>0$, and hence $\Phi_P(\alpha)>0$.
\end{proof}

Proposition~\ref{prop:endpointall} contains Corollary~\ref{cor:endpoint}. Since
$L$ is zero-diminishing, real-rootedness of $\LP$ also forces that of $P$. The
Hermite polynomials of Proposition~\ref{prop:hermite} show that the converse
fails.

\begin{remark}\label{rem:consistency}
Lemma~\ref{lem:signs} is a degree-two instance of the zero-diminishing property
quoted in Section~\ref{sec:structural}: were $\disc P<0<\disc\LP$ possible, $P$
would have no real zero while $\LP$ had two, so $L$ would have increased the
number of real zeros. The collision analysis therefore reproves that property
independently in degree two.
\end{remark}

For $P(x)=(x-1)(x-2)=x^{2}-3x+2$ one has $\disc P=1$, $\disc\LP=-7$, and $t=9/2$, so $\alpha_0=2/7$ and $x_0=12/7$. Here
$\Delta_\alpha=(\alpha+2)(2-7\alpha)$, so writing $\alpha=2/7-\delta$ gives
$\Delta_\alpha=16\delta-7\delta^{2}$ and each branch differs from the midpoint
by
\[
  \frac{\sqrt{\Delta_\alpha}}{4}
  =\sqrt{\delta}-\frac{7}{32}\delta^{3/2}+O(\delta^{5/2}).
\]
The collision is therefore an ordinary square-root branch point, and no Puiseux
machinery is needed in the quadratic case: the explicit root formula supplies
the exponent directly.

\begin{remark}[Comparison with the root locus of control theory]\label{rem:rootlocus}
The phenomena of this section have close analogues in Evans's root locus method
\cite{Evans1948}. There, the orders at which two real zeros meet and leave the
real axis are the \emph{breakaway points}, detected by the vanishing of a
discriminant. At large parameter the locus is organized by asymptotes with a
common centroid. Theorem~\ref{thm:collision} plays the role of the breakaway
computation, Remark~\ref{rem:puiseux} that of the square-root departure from the
axis, and $\CP(\lambda)$ that of the centroid. Since the classical root locus varies a
linear pencil while \eqref{eq:Galpha} is a gamma-weighted deformation, the
correspondence is one of structure and vocabulary rather than of theorems.
\end{remark}

\begin{remark}[An exact barycenter law]\label{rem:barycentre}
For $m=2$ the mean of the two zeros of $G_\alpha$ is
$-b_1(\alpha)/(2b_2)=-a_1(\alpha+2)/(4a_2)$, which is \emph{exactly} affine in
$\alpha$. It starts at $-a_1/(2a_2)$, the mean of the zeros of $P$, and moves
with constant velocity $-a_1/(4a_2)$, which is precisely the mean of the two
zeros of $\LP$. So for the average of the branches the asymptotic velocity law
of Corollary~\ref{cor:continuous} holds identically at every order, with no
error term. Exact barycenter identities of this type, in the
fractional-derivative direction, are obtained in \cite{WaniSheikh2026}.
\end{remark}

\subsection{Arbitrary degree}\label{subsec:gencoll}
Although no closed form for the collision orders should be expected beyond
degree two, the structure responsible for Theorem~\ref{thm:collision} persists
in every degree.

\begin{proposition}\label{prop:discpoly}
Let $\deg P=m\ge2$. Then $\Delta_\alpha=\disc_x G_\alpha$ is a polynomial in
$\alpha$ of degree at most $m(m-1)$, with
\begin{equation}\label{eq:endpoints}
  \Delta_0=(m!)^{2m-2}\disc P,
  \qquad
  [\alpha^{m(m-1)}]\,\Delta_\alpha=\disc\LP.
\end{equation}
The degree is exactly $m(m-1)$ if and only if $\disc\LP\neq0$.
\end{proposition}

\begin{proof}
The discriminant of a degree-$m$ polynomial is a polynomial in its
coefficients, and each $b_j$ is a polynomial in $\alpha$, so $\Delta_\alpha$ is
a polynomial in $\alpha$.

At $\alpha=0$ we have $b_j(0)=a_j\,j!\prod_{i=j+1}^{m}i=a_j\,j!\,(m!/j!)=m!\,a_j$,
so $G_0=m!\,P$. For a degree-$m$ polynomial $p$ and a constant $c\neq0$ one has
$\disc(cp)=c^{2m-2}\disc(p)$, because $\disc(p)=p_m^{2m-2}\prod_{i<j}(r_i-r_j)^2$
and scaling changes only the leading coefficient. Hence
$\Delta_0=(m!)^{2m-2}\disc P$.

For the leading coefficient, assign $b_j$ the weight $m-j$. The discriminant
is then weighted homogeneous of total weight $m(m-1)$. Indeed, replacing $b_j$
by $t^{m-j}b_j$ produces $t^{m}p(x/t)$, whose zeros are $t$ times those of $p$
and whose leading coefficient is unchanged. By
$\disc(p)=p_m^{2m-2}\prod_{i<j}(r_i-r_j)^2$, the discriminant is therefore
multiplied by $t^{2\binom{m}{2}}=t^{m(m-1)}$. (For discriminants and their
homogeneity properties see \cite[Ch.~12]{GKZ1994}.) Since $b_j$ has degree at
most $m-j$ in $\alpha$, with $a_j\,j!$ as the coefficient of $\alpha^{m-j}$, the
coefficient of $\alpha^{m(m-1)}$ in $\Delta_\alpha$ is obtained by replacing
each $b_j$ by $a_j\,j!$. That substitution turns $\disc_x\sum_j b_jx^{j}$ into
$\disc_x\sum_j a_j\,j!\,x^{j}=\disc\LP$.
\end{proof}

Since $\Gamma(\alpha+m+1)\neq0$ for $\alpha>0$, collisions at positive order
are exactly the positive zeros of $\Delta_\alpha$. Proposition~\ref{prop:discpoly}
therefore yields an existence criterion valid in every degree.

\begin{corollary}\label{cor:gencoll}
Let $\deg P=m\ge2$ with $\disc\LP\neq0$. There are at most $m(m-1)$ collision
orders in total. Suppose now that $P$ is real and that $\disc P$ and
$\disc\LP$ are nonzero. If they have opposite signs, then at least one
collision occurs at positive order, and the number of positive collision
orders, counted with multiplicity, is odd. If they have the same sign, that
number is even, possibly zero.
\end{corollary}

\begin{proof}
The count is the degree bound of Proposition~\ref{prop:discpoly}. By
\eqref{eq:endpoints} the sign of $\Delta_0$ is that of $\disc P$ and the sign
of $\Delta_\alpha$ for large $\alpha$ is that of $\disc\LP$, since
$(m!)^{2m-2}>0$. A real polynomial has an odd number of zeros, with
multiplicity, in $(0,\infty)$ precisely when its values at $0$ and at
$+\infty$ have opposite signs.
\end{proof}

The bound $m(m-1)$ is not sharp, because $\Delta_\alpha$ carries factors
supported at negative $\alpha$.

\begin{proposition}\label{prop:spurious}
Let $\deg P=m\ge2$. Then
\begin{equation}\label{eq:spurious}
  \Delta_\alpha=\Bigl[\prod_{i=2}^{m}(\alpha+i)^{\,i-1}\Bigr]\Phi_P(\alpha),
\end{equation}
where $\Phi_P$ is a polynomial in $\alpha$ of degree at most $m(m-1)/2$ with
\[
  \Phi_P(0)=\kappa_m\disc P,\quad
  \kappa_m=\frac{(m!)^{2m-2}}{\prod_{i=2}^{m}i^{\,i-1}}>0,
  \qquad
  [\alpha^{m(m-1)/2}]\,\Phi_P=\disc\LP.
\]
Consequently, if $\Phi_P\not\equiv0$ (for instance if $\disc P\neq0$ or
$\disc\LP\neq0$), at most $m(m-1)/2$ collisions occur at positive order. The
universal zeros of $\Delta_\alpha$ are the $m-1$ distinct negative orders
$\alpha=-2,\dots,-m$, at which $\Delta_\alpha$ vanishes to order at least
$1,2,\dots,m-1$ respectively, with equality for generic $P$; the universal
factor has total degree $\sum_{i=2}^{m}(i-1)=m(m-1)/2$, exactly half the bound
of Proposition~\ref{prop:discpoly}.
\end{proposition}

\begin{proof}
Fix $i$ with $2\le i\le m$ and put $\varepsilon=\alpha+i$, working over the ring
$\mathbb{C}[[\varepsilon]]$ of formal power series, a complete discrete
valuation ring with uniformizer $\varepsilon$ and residue characteristic zero;
write $\operatorname{ord}$ for its valuation. The valuation statement is
established first for coefficient vectors in a Zariski-open subset of
coefficient space, the conditions defining it being collected below; the
passage to arbitrary $a$ is by polynomial continuation at the end of the proof.
Assume accordingly for the moment that $a_0\neq0$ and $a_i\neq0$.

By \eqref{eq:Galpha} the factor $\alpha+i$ occurs in $b_j$ exactly when
$j+1\le i\le m$, that is, when $j\le i-1$. Hence
\[
  G_\alpha(x)=x^{i}V(x,\varepsilon)+\varepsilon A(x,\varepsilon),
\]
where $V(x,\varepsilon)=\sum_{j\ge i}b_j(\alpha)x^{j-i}$ satisfies
$V(0,0)=b_i(-i)\neq0$, and $A(x,\varepsilon)=\varepsilon^{-1}\sum_{j<i}b_j(\alpha)x^{j}$
has degree at most $i-1$ in $x$ with
$A(0,0)=a_0\prod_{\ell=1,\ \ell\neq i}^{m}(\ell-i)\neq0$.

Reducing modulo $\varepsilon$ gives $G\equiv x^{i}V(x,0)$ with $V(0,0)\neq0$, so
Hensel's lemma factors $G$ over $\mathbb{C}[[\varepsilon]]$ as $G=u\,E_iW$ with
$u$ a unit, $E_i$ monic of degree $i$ satisfying $E_i\equiv x^{i}\pmod\varepsilon$,
and $W$ of degree $m-i$ with $W(0,\cdot)$ a unit. Write
$E_i=x^{i}+v_{i-1}x^{i-1}+\dots+v_0$. The congruence gives
$\operatorname{ord}v_k\ge1$ for every $k<i$, and evaluating the factorization at
$x=0$ gives $u(0,\varepsilon)\,v_0\,W(0,\varepsilon)=G_\alpha(0)=\varepsilon A(0,\varepsilon)$,
so $\operatorname{ord}v_0=1$. Thus $E_i$ is an Eisenstein polynomial over
$\mathbb{C}[[\varepsilon]]$.

The valuation of $\disc E_i$ can now be computed directly from the Eisenstein
property. Since $\mathbb{C}$ is algebraically closed of characteristic zero, the
algebraic closure of $\mathbb{C}((\varepsilon))$ is the field of Puiseux series
in $\varepsilon$. On it, $\operatorname{ord}$ is the exponent of the lowest
occurring term. Let $\theta_1,\dots,\theta_i$ be the zeros of $E_i$ there, and
write $v_0=\varepsilon\eta$ with $\eta(0)\neq0$. Substituting $x=\varepsilon^{1/i}y$
and dividing by $\varepsilon$ turns $E_i$ into the monic polynomial
\begin{equation}\label{eq:rescaled}
  \varepsilon^{-1}E_i(\varepsilon^{1/i}y)=\prod_{j=1}^{i}(y-y_j)
  =y^{i}+\sum_{k=0}^{i-1}v_k\,\varepsilon^{k/i-1}y^{k},
  \qquad y_j=\theta_j\,\varepsilon^{-1/i},
\end{equation}
whose $k$-th coefficient has valuation
$\operatorname{ord}v_k-1+k/i\ge k/i\ge0$, with equality to $0$ only at $k=0$,
where that coefficient is $\eta$. The coefficients of \eqref{eq:rescaled} are
therefore integral, and hence so are the $y_j$: if $\operatorname{ord}y_j<0$
then $\operatorname{ord}y_j^{\,i}=i\operatorname{ord}y_j<\operatorname{ord}
(c\,y_j^{\,k})$ for every $k<i$ and every integral $c$, so \eqref{eq:rescaled}
could not vanish at $y_j$. Reduction modulo the maximal ideal is a ring
homomorphism on the integral elements, so applying it to \eqref{eq:rescaled}
gives
\[
  \prod_{j=1}^{i}(y-\bar y_j)=y^{i}+\eta(0),\qquad \bar y_j=y_j\bmod\varepsilon .
\]
Because $\eta(0)\neq0$ and the characteristic is zero, the derivative $iy^{i-1}$ of
$y^{i}+\eta(0)$ vanishes only at $y=0$, which is not a zero of $y^{i}+\eta(0)$; that
polynomial is therefore separable, and the $\bar y_j$ are distinct and nonzero.
Consequently $\operatorname{ord}y_j=0$ and $\operatorname{ord}(y_j-y_k)=0$ for
$j\neq k$, that is, $\operatorname{ord}\theta_j=1/i$ and
$\operatorname{ord}(\theta_j-\theta_k)=1/i$. Since $E_i$ is monic,
$\disc E_i=\prod_{j<k}(\theta_j-\theta_k)^{2}$, so
\[
  \operatorname{ord}\disc E_i
  =\sum_{j\neq k}\operatorname{ord}(\theta_j-\theta_k)
  =i(i-1)\cdot\tfrac1i=i-1 .
\]
The exponent is exactly $i-1$ because $i$ is invertible in $\mathbb{C}$, which
is the elementary content of the tameness of the extension that $E_i$ generates.
Finally, up to a unit,
\[
  \disc(G)=\disc(E_i)\,\disc(W)\,\operatorname{Res}(E_i,W)^{2},
\]
and both $\disc(W)$ and $\operatorname{Res}(E_i,W)$ are units. The first is a unit because
$W$ reduces modulo $\varepsilon$ to a nonzero constant multiple of $V(x,0)$,
which has simple zeros for generic $a$. The second is a unit because the zeros of $E_i$
reduce to $0$ modulo $\varepsilon$ while those of $W$ do not. Hence
$\operatorname{ord}_\varepsilon\Delta_\alpha=i-1$, so $\Delta_\alpha$ vanishes at
$\alpha=-i$ to order exactly $i-1$.

The computation above used three conditions on $a$: $a_0\neq0$ and
$a_i\neq0$, so that $A(0,0)\neq0$ and $V(0,0)\neq0$; and $\disc_xV(x,0)\neq0$,
so that $\disc W$ is a unit. The unit property of $\operatorname{Res}(E_i,W)$
needs nothing further. Each condition is the nonvanishing of a polynomial in
$a$, and none of these polynomials vanishes identically. For the first two this
is clear. For the third, the coefficient of $x^{j-i}$ in $V(x,0)$ is
$b_j(-i)=a_j\,j!\prod_{\ell=j+1}^{m}(\ell-i)=(m-i)!\,a_j\,j!/(j-i)!$, so
$V(x,0)=(m-i)!\,P^{(i)}(x)$, and the third condition says that $P^{(i)}$ has
simple zeros. Since $a_i,\dots,a_m$ can be chosen to make $P^{(i)}$ any
polynomial of degree $m-i$, it holds for generic $a$. Since coefficient
space is irreducible, finitely many nonempty Zariski-open subsets meet. The
common nonvanishing locus of these conditions for $i=2,\dots,m$ is therefore a
nonempty Zariski-open, hence dense, subset, and on it $\prod_{i=2}^{m}(\alpha+i)^{i-1}$ divides $\Delta_\alpha$. Both $\Delta_\alpha$ and the remainder of
its division by that product have coefficients polynomial in $a$, so the
remainder vanishes identically, and \eqref{eq:spurious} holds for all $a$. The
degree count is
$m(m-1)-\sum_{i=2}^{m}(i-1)=m(m-1)-\tfrac{m(m-1)}{2}=\tfrac{m(m-1)}{2}$, and
the two displayed identities follow from \eqref{eq:endpoints} on evaluating the
product at $\alpha=0$ and comparing leading coefficients, the product being
monic. Finally, the factors $\alpha+i$ with $2\le i\le m$ are positive for
$\alpha>0$, so positive collisions are positive zeros of $\Phi_P$.
\end{proof}

For $m=2$ this recovers Lemma~\ref{lem:discfactor}: the product is $\alpha+2$,
$\Phi_P$ is affine, $\kappa_2=2$, and the single zero is $\alpha_0$. For
$m=3$ and $m=4$ one finds $\kappa_3=72$ and $\kappa_4=165888$, and
$\deg\Phi_P=3$ and $6$ respectively; these values, and identities
\eqref{eq:endpoints} and \eqref{eq:spurious}, were verified by symbolic
computation.

The universal factor is the one familiar from Laguerre polynomials. If $\LP=c\,(y-\lambda)^{m}$ with $\lambda\neq0$, then $G_\alpha(x)$ is a constant multiple of
$L_m^{(\alpha)}(x/\lambda)$, and \eqref{eq:spurious} reduces to Stieltjes's
formula $\disc L_m^{(\alpha)}=\prod_{j=1}^{m}j^{\,j-2m+2}(j+\alpha)^{j-1}$
\cite[(6.71.6)]{Szego1975}, with $\Phi_P$ constant. What
Proposition~\ref{prop:spurious} adds is that the same factor divides
$\Delta_\alpha$ for every $P$.

\begin{remark}[Negative orders are derivatives]\label{rem:negorder}
The proof shows more than the valuation. By \eqref{eq:Galpha},
\[
  G_{-i}(x)=(m-i)!\,x^{i}P^{(i)}(x)\qquad(1\le i\le m),
\]
as one expects from $J^{-i}=d^{i}/dx^{i}$: as $\alpha\to-i$, $m-i$ zeros of
$G_\alpha$ tend to the zeros of $P^{(i)}$. If $a_0a_i\neq0$, the other $i$ tend
to $0$ like $(\alpha+i)^{1/i}$ and, since $E_i$ is Eisenstein, a small loop of
$\alpha$ around $-i$ permutes them in a single $i$-cycle. For $i=m$ all $m$
zeros collapse. So if $a_0\neq0$ and $m\ge2$, every scaled branch $y(h)=h\,r(1/h)$, continued
from $h=0$ toward $h=-1/m$, meets a singularity no later than $-1/m$. The
all-orders expansion in $h=1/\alpha$ of Remark~\ref{rem:allorders} therefore has
radius of convergence at most $1/m$. Its open disc of convergence can therefore
cover positive orders only with $\alpha>m$; convergence at the boundary $\alpha=m$
is a separate question, and a nearer singularity may impose a larger threshold. (For $m=1$ the branch $y=-(a_0/a_1)(1+h)$ is entire.) A given branch may
have a smaller radius, set by a collision at some $|h|<1/m$. For
$\LP=(y-1)(y-2)$ the cleared discriminant is $(1+2h)(1+10h)$, and the
branches meet at $h=-1/10$. When $h=-1/m$ is the only singularity of the
branch on its circle of convergence, the local Puiseux expansion there gives
coefficients of size $K(-m)^{n}n^{-1-1/m}$. With several singularities on that
circle their contributions add and may cancel: for $\LP=2y^{2}+\sqrt6\,y+1$
the branch points are $h=\pm\frac12$ and the odd coefficients beyond the first
vanish. For $\LP=y^{2}+y-2$, the ratios of successive
coefficients tend to $-2$, with fitted power $n^{-1.49}$. The local monodromy at the orders $\alpha=-i$ has a consequence for the zero branches $r_1(\alpha),\dots,r_m(\alpha)$ as functions of $\alpha$:
\emph{if $a_0a_{m-1}a_m\neq0$, they are linearly independent over $\mathbb{C}$}, so the smallest homogeneous scalar linear differential equation with
coefficients in $\mathbb{C}(\alpha)$ that they all satisfy has order $m$. Such
an equation exists: the operator
$f\mapsto W(r_1,\dots,r_m,f)/W(r_1,\dots,r_m)$, with $W$ the Wronskian, has
coefficients that are invariant under monodromy and algebraic, hence
rational. Indeed, the loops around $-m$ and
$-(m-1)$ give an $m$-cycle and an $(m-1)$-cycle fixing one branch, the latter
because $P^{(m-1)}$ has the single zero $-a_{m-1}/(ma_m)\neq0$. The monodromy group is therefore $2$-transitive. Its permutation
representation is thus the sum of the trivial representation and an irreducible
one. A space of constant-coefficient relations $\sum_ic_ir_i\equiv0$ is
invariant, so it is $0$, the constant vectors,
the complementary subspace of vectors with zero sum, or everything. The last
two would force the branches to coincide or to vanish. The constant vectors
would mean $\sum_ir_i\equiv0$, whereas
$\sum_ir_i=-b_{m-1}(\alpha)/b_m=-a_{m-1}(\alpha+m)/(ma_m)$. If $a_{m-1}=0$ that
relation does hold, and further relations arise when $P$ is a polynomial in
$x^{d}$ for some $d\ge2$.
\end{remark}

\subsection{Degeneration to the zero-root locus}\label{subsec:degeneration}

Section~\ref{sec:geometry} showed that for fixed $P$ the origin is pinned. The
complementary question is what happens when a \emph{nonzero} transformed zero
is driven to the origin by varying $P$, that is, as the coefficient vector
approaches the locus $\{a_0=0\}$. Since $\LP(0)=a_0$ and $\LP'(0)=a_1$, a
simple transformed zero near the origin satisfies
$\lambda=-a_0/a_1+O(a_0^{2})$ when $a_1\neq0$, so $\lambda\to0$ exactly as
$a_0\to0$. The next proposition shows that this degeneration is visible in the
collision structure: a coefficient-dependent discriminant value migrates to the
universal one at $\alpha=-2$.

\begin{proposition}\label{prop:degeneration}
Let $\deg P=m\ge2$ with $a_0=0$. Then $(\alpha+2)^{2}$ divides
$\Delta_\alpha$, so that $\Phi_P(-2)=0$: the multiplicity of $\alpha=-2$ rises
from its generic value $1$ to at least $2$. Consequently, let
$P_\varepsilon$ be a continuous family of degree exactly $m$, with
$a_m(\varepsilon)$ and $\disc L[P_\varepsilon]$ bounded away from $0$,
$a_0(\varepsilon)\to0$, and $a_1(0)\neq0$. Then at least one zero of
$\Phi_{P_\varepsilon}$ tends to $-2$ as $\varepsilon\to0$. The hypothesis on
$\disc L[P_\varepsilon]$ keeps $\deg\Phi_{P_\varepsilon}$ fixed at $m(m-1)/2$,
so that no zero escapes to infinity.
\end{proposition}

\begin{proof}
If $a_0=0$ then $b_0\equiv0$ in \eqref{eq:Galpha}, so $G_\alpha(x)=x\,R_\alpha(x)$
with $R_\alpha(x)=\sum_{j\ge1}b_j(\alpha)x^{j-1}$ and
$R_\alpha(0)=b_1(\alpha)=a_1\prod_{i=2}^{m}(\alpha+i)$. For a polynomial
factored as $x\,R(x)$ one has $\disc(xR)=R(0)^{2}\disc(R)$, so
$\Delta_\alpha$ carries the factor
$b_1(\alpha)^{2}=a_1^{2}\prod_{i=2}^{m}(\alpha+i)^{2}$. At $i=2$ the exponent
$2$ exceeds the universal exponent $i-1=1$ of
Proposition~\ref{prop:spurious}, so $(\alpha+2)$ divides $\Phi_P$. The final
statement follows because the zeros of a polynomial depend continuously on its
coefficients \cite[Ch.~1]{Marden1966}.
\end{proof}

In degree two this is explicit: the expression $-2\disc P/\disc\LP$ of
\eqref{eq:alpha0} equals $-2$ precisely when $a_0=0$, since then $\disc P=\disc\LP=a_1^{2}$, and
\eqref{eq:discfactor} degenerates to $\Delta_\alpha=a_1^{2}(\alpha+2)^{2}$. For
the model family $P_\varepsilon(x)=x^{2}-x+\varepsilon$ one finds
$\lambda(\varepsilon)=\varepsilon+O(\varepsilon^{2})$ for the small transformed
zero and $\alpha_0(\varepsilon)=-2-8\varepsilon+O(\varepsilon^{2})$ for the
collision order, so both degenerations occur at the same rate.

The locus $\{a_0=0\}$ is therefore where two phenomena coincide: a transformed
zero reaches the rigid origin, and a coefficient-dependent discriminant value of the
order projection merges with a universal one. Driving $\lambda\to0$ through
variation of $P$ is thus a collision of discriminant values. When $a_1\neq0$ the joint implicit function theorem in $(\varepsilon,h)$
already describes the moving zero (Section~\ref{sec:limitations}); the
degenerate cases need a genuine two-parameter unfolding, which we do not carry
out here.

\begin{remark}[The generic local model at a collision]\label{rem:puiseux}
Away from such degenerations the local picture is the standard one. Write
$Q(h,y)$ for $Q_h(y)$. If $Q(h_0,y_0)=\partial_yQ(h_0,y_0)=0$ while
$\partial_y^{2}Q(h_0,y_0)\neq0$ and $\partial_hQ(h_0,y_0)\neq0$, then Weierstrass preparation gives
\[
  y(h)=y_0\pm c\,(h-h_0)^{1/2}+O(h-h_0),\qquad c\neq0,
\]
so a simple two-sheet collision is an ordinary square-root branch point
\cite[Ch.~IV]{Walker1950}. The explicit computation for $P=(x-1)(x-2)$ in
Section~\ref{subsec:quadcoll} realizes this model. Collisions of more than two
sheets require other fractional exponents, governed by the local Newton
polygon. For real coefficients $Q_{\bar h}(\bar y)=\overline{Q_h(y)}$, so
nonreal zeros occur in conjugate pairs and a pair of simple real zeros can
leave the real axis only by meeting at a real point of the discriminant locus.
\end{remark}

\subsection{The implicit root-velocity equation}
Along any simple reduced root branch we have $H(\alpha,r(\alpha))=0$, and
implicit differentiation gives $dr/d\alpha=-H_\alpha(\alpha,r)/H_x(\alpha,r)$. Since
$\frac{d}{dz}\frac{1}{\Gamma(z)}=-\frac{\psi(z)}{\Gamma(z)}$ with
$\psi=\Gamma'/\Gamma$ \cite[\S5.2]{DLMF}, we obtain
\begin{equation}\label{eq:ode}
  \frac{dr}{d\alpha}
  =\frac{\displaystyle\sum_{j=0}^{m}
     a_j\,j!\,\psi(j+\alpha+1)\,r^{j}/\Gamma(j+\alpha+1)}
        {\displaystyle\sum_{j=1}^{m}
     j\,a_j\,j!\,r^{j-1}/\Gamma(j+\alpha+1)}.
\end{equation}
Along a root branch this identity is valid precisely while
$H_x(\alpha,r)\neq0$, that is, away from the collisions of Proposition~\ref{prop:spurious}.
Corollary~\ref{cor:continuous} supplies its large-$\alpha$ asymptotics, and
Remark~\ref{rem:mean} identifies its mean over a unit interval with the
discrete spacing.

\section{The root locus as a plane curve}\label{sec:curve}

The objects assembled in Sections~\ref{sec:parameter}--\ref{sec:collision} are
the local geometry of a single global one. Recall $G_\alpha$ from
\eqref{eq:Galpha} and regard $\alpha$ and $x$ as two complex coordinates.

\subsection{The curve and its points at infinity}

\begin{proposition}\label{prop:curve}
Let $\deg P=m\ge1$. Then $G(\alpha,x):=G_\alpha(x)$ is a polynomial in
$(\alpha,x)$ of total degree exactly $m$, its Newton polygon is contained in
the triangle $\{(a,j):a,j\ge0,\ a+j\le m\}$, and its degree-$m$ homogeneous
part is
\begin{equation}\label{eq:leadingform}
  F_m(\alpha,x)=\sum_{j=0}^{m}a_j\,j!\,\alpha^{m-j}x^{j}
  =\alpha^{m}\,\LP(x/\alpha).
\end{equation}
Consequently the points at infinity of the affine curve
$V=\{G=0\}\subset\mathbb{C}^{2}$ are exactly the distinct directions $[1:\lambda]$,
one for each distinct zero $\lambda$ of $\LP$; in projective coordinates
$[\alpha:x:z]$ these are the points $[1:\lambda:0]$. There is no additional
point in the direction $\alpha=0$. For a simple zero $\lambda$ the projective
closure meets the line at infinity transversally at $[1:\lambda:0]$, so that
point is nonsingular and carries a single smooth branch. At a multiple zero, the intersection multiplicity with the line
at infinity and the number of branches of the normalization above that point
need not agree; we make no claim about the latter.
\end{proposition}

\begin{proof}
By \eqref{eq:Galpha}, $b_j(\alpha)=a_j\,j!\prod_{i=j+1}^{m}(\alpha+i)$ has
degree at most $m-j$ in $\alpha$, with equality when $a_j\neq0$. Every monomial of $G$ therefore has the form
$\alpha^{a}x^{j}$ with $a\le m-j$, which gives the Newton polygon statement and
total degree at most $m$. The monomial $a_m\,m!\,x^{m}$ has degree $m$, so the
degree is exactly $m$. The degree-$m$ part collects, for each $j$, the leading
$\alpha$-term $a_j\,j!\,\alpha^{m-j}$ of $b_j$, which is
\eqref{eq:leadingform}. The directions at infinity of a plane curve are the zeros of its leading form.
Since $F_m(0,1)=a_m\,m!\neq0$, the direction $[0:1]$ is not among them, and on
$\alpha\neq0$ the equation $F_m(\alpha,x)=0$ reads $\LP(x/\alpha)=0$. In the chart
$\alpha=1$ at infinity the curve is $g(x,z)=\LP(x)+z\,F_{m-1}(1,x)+\dots=0$, so at a
simple zero $g_x(\lambda,0)=\LP'(\lambda)\neq0$; hence $z$ is a local parameter on the
curve and its intersection with $z=0$ is transverse.
\end{proof}

Thus the entire family of root trajectories, over all complex orders at once,
lies on one affine plane curve of degree $m$ in the coordinates $(\alpha,x)$,
counted with multiplicity. The curve is defined over $\mathbb{Q}$ whenever $P$
is. (For $P=x^m$, $G=m!\,x^m$ and the curve is the line $x=0$ taken $m$
times.) Treating the order as a second indeterminate is not new. Galligo does this
for fractional \emph{derivatives}: he forms a bivariate polynomial $P_f(x,q)$ in
the variable and the order, and iterates the resulting homotopy to associate to
$f$ a plane spline curve he calls the \emph{stem} of $f$ \cite{Galligo2013}. What follows is the corresponding
picture in the integration direction, where the leading form turns out to be the
transform $\LP$ of Borwein, Chen, and Dilcher: the classical scaling theorem is
the statement that $V$ has the asymptotic directions prescribed by that form.

If $a_0\neq0$, then $G$ is irreducible: it is Eisenstein at the prime
$\alpha+m$ of $\mathbb{C}[\alpha]$, since every coefficient below the top
contains $\alpha+m$ and the constant contains it once. With $E=x\,d/dx$ it
satisfies the shift identity
$(\alpha+E)G(\alpha,x)=(\alpha+m)G(\alpha-1,x)$, the curve's form of
$\frac{d}{dx}I^{k}P=I^{k-1}P$. When $a_0\neq0$, the negative fibers
$G(-\ell,x)=(m-\ell)!\,x^{\ell}P^{(\ell)}(x)$ of Remark~\ref{rem:negorder}
supply cycles of length $m$, $m-1$ (when $a_{m-1}\neq0$) and $2$ (when
$a_2\neq0$ and $P''$ is squarefree) in the monodromy of the projection to $\alpha$, which under these
conditions is the full symmetric group: the $m$-cycle gives transitivity, the
$(m-1)$-cycle fixing a point gives two-transitivity, and a two-transitive group
containing a transposition contains all transpositions. For $P$ with rational coefficients satisfying these conditions, the Galois
group of $G$ over $\mathbb{Q}(\alpha)$ is then $S_m$, and by Hilbert's
irreducibility theorem $G(\alpha_0,x)$ has Galois group $S_m$ over $\mathbb{Q}$
for all integers $\alpha_0$ outside a thin set. In the case
$\LP=c\,(y-\lambda)^m$ these fibers are generalized Laguerre polynomials, whose
Galois groups have been studied since Schur, more recently by Newton-polygon
methods; see Hajir \cite{Hajir2005} and the references there.

\subsection{The constant term is an asymptote offset}

The next homogeneous piece of $G$ upgrades a direction to a line, and
reproduces $\CP$ without any reference to Theorem~\ref{thm:branch}.

\begin{proposition}\label{prop:asymptote}
Let $\lambda\neq0$ be a simple zero of $\LP$. Then $V$ has an asymptote in the
direction $\lambda$, namely the line $x=\lambda\alpha+c$ with
\[
  c=\frac{\sum_{j}a_j\,j!\,\sigma_1(j)\lambda^{j}}{\LP'(\lambda)}
   =\frac{R_P(\lambda)}{2\LP'(\lambda)}=\CP(\lambda).
\]
\end{proposition}

\begin{proof}
Write $\prod_{i=j+1}^{m}(\alpha+i)
=\alpha^{m-j}+\bigl(\sigma_1(m)-\sigma_1(j)\bigr)\alpha^{m-j-1}
+O(\alpha^{m-j-2})$, since the coefficient of the second term is
$\sum_{i=j+1}^{m}i$. Substituting $x=\lambda\alpha+c$ and using
$(\lambda\alpha+c)^{j}=\lambda^{j}\alpha^{j}+j\lambda^{j-1}c\,\alpha^{j-1}
+O(\alpha^{j-2})$, the coefficient of $\alpha^{m}$ in $G(\alpha,\lambda\alpha+c)$
is $\sum_j a_j\,j!\,\lambda^{j}=\LP(\lambda)=0$, and the coefficient of
$\alpha^{m-1}$ is
\[
  c\sum_j j\,a_j\,j!\,\lambda^{j-1}
  +\sigma_1(m)\sum_j a_j\,j!\,\lambda^{j}
  -\sum_j a_j\,j!\,\sigma_1(j)\lambda^{j}
  =c\,\LP'(\lambda)-\tfrac12 R_P(\lambda),
\]
the middle sum vanishing because $\LP(\lambda)=0$ and the last equality
following from $\sigma_1(j)=j(j+1)/2$ and \eqref{eq:RS}. Setting this to zero
gives $c$, and $\LP'(\lambda)\neq0$ by simplicity.
\end{proof}

So the line $\ell_\lambda$ of Corollary~\ref{cor:geometry} is not merely
analogous to an asymptote: it is the asymptote of $V$ in the classical sense,
with offset $\CP(\lambda)$ from the origin. The derivation is by
elementary plane-curve geometry rather than by perturbation of a holomorphic
branch.

\subsection{Collisions as discriminant values}

When $\Delta_\alpha\not\equiv0$, the projection $\pi(\alpha,x)=\alpha$ is an
$m$-sheeted finite map from $V$ to the $\alpha$-line, with $\Delta_\alpha$ its
discriminant. The curve $V$ is irreducible when $a_0\neq0$, by the
Eisenstein argument above, but not when $a_0=0$; we do not assert that it is
nonsingular, and a zero of
$\Delta_\alpha$ need not be a branch value of the normalization. We therefore
speak of universal \emph{discriminant} values below, and we reserve
\emph{branch point} for places where the ramification has been established, as
in Remark~\ref{rem:puiseux}. The collisions of
Section~\ref{sec:collision} are exactly the zeros of the discriminant.
Proposition~\ref{prop:spurious} identifies $m-1$ distinct universal
discriminant values, present for every $P$ and located at
$\alpha=-2,\dots,-m$. Together they account for a universal factor of
$\Delta_\alpha$ of total degree $m(m-1)/2$. They are the orders at which the
low-order coefficients of $G_\alpha$ vanish and sheets collide, generically at
the origin. This vanishing of low-order coefficients is why the factor is universal and
coefficient-free. The remaining discriminant values, the zeros of $\Phi_P$,
depend on $P$. Circling a simple branch point transposes the two sheets meeting there,
which is the monodromy reason why the two real zeros of
Section~\ref{subsec:complex} re-emerge as a complex-conjugate pair rather than
as two independent real branches.

\begin{remark}[Genus, and why degree two is special]\label{rem:genus}
Assume $a_0\neq0$, so that $V$ is irreducible. By Baker's theorem on Newton
polygons \cite{Baker1893}, the geometric genus of $V$ is at most the number of interior lattice points of its Newton polygon.
That the bound is attained for nondegenerate coefficients is due to Khovanskii
\cite{Khovanskii1978}. For the triangle of
Proposition~\ref{prop:curve} that count is
\[
  \#\{(a,j):a,j\ge1,\ a+j\le m-1\}=\binom{m-1}{2}=\frac{(m-1)(m-2)}{2}.
\]
For $m=2$ the bound is $0$: the trajectory curve is a conic, hence rational.
This rationality is consistent with the simplicity of the quadratic case, although the closed
form of Section~\ref{subsec:quadcoll}, with $\alpha_0$ a ratio of two
discriminants, comes more specifically from the reduced discriminant being affine
in $\alpha$. For $m=3$ the bound is $1$, so the smooth model is rational or
elliptic, and it is elliptic when the Newton polygon is the full triangle and
the coefficients are nondegenerate with respect to it; for $m=4$ the bound is $3$. This growth of the genus bound is consistent with, though it does not by itself imply, the absence of a
comparably simple closed form in higher degree.
We have not verified the singularity structure of $V$, so the genus is asserted
here only as the generic bound.
\end{remark}

\subsection{The other end of the curve}\label{subsec:zeroend}

Sections~\ref{sec:expansion}--\ref{sec:geometry} describe $V$ near
$\alpha=\infty$. The opposite end is $\alpha=0$, where $J^{0}$ is the identity
and the fiber of $\pi$ consists of the zeros of $P$ itself. The first-order
behavior there is governed by a third weighted transform, this
time with harmonic rather than arithmetic weights. Write $H_j=\sum_{i=1}^{j}1/i$
with $H_0=0$, and set
\begin{equation}\label{eq:Wdef}
  W[P](y)=\sum_{j=0}^{m}a_j H_j\,y^{j}.
\end{equation}

\begin{proposition}[Initial root velocity]\label{prop:initvel}
Let $\rho$ be a simple zero of $P$. Then the branch of $V$ through
$(0,\rho)$ is holomorphic near $\alpha=0$ and
\begin{equation}\label{eq:initvel}
  \frac{dr}{d\alpha}\bigg|_{\alpha=0}=\frac{W[P](\rho)}{P'(\rho)}.
\end{equation}
In particular the Euler--Mascheroni constant, which appears in the individual
terms, cancels from the result.
\end{proposition}

First-order deformation formulas for simple zeros at the endpoint orders are
obtained in the derivative direction by Wani and Sheikh
\cite{WaniSheikh2026}; Proposition~\ref{prop:initvel} is the order-zero
statement in the integration direction, and no priority is claimed for the
mechanism.

\begin{proof}
Since $\Gamma(j+1)=j!$, we have $H(0,x)=P(x)$, so $H_x(0,\rho)=P'(\rho)\neq0$
and the implicit function theorem applies. Evaluating \eqref{eq:ode} at
$\alpha=0$ and using $\psi(j+1)=-\gamma+H_j$ \cite[\S5.4]{DLMF},
\[
  \frac{dr}{d\alpha}\bigg|_{\alpha=0}
  =\frac{\sum_j a_j\,\psi(j+1)\rho^{j}}{P'(\rho)}
  =\frac{-\gamma P(\rho)+W[P](\rho)}{P'(\rho)},
\]
and $P(\rho)=0$.
\end{proof}

For $P(x)=(x-1)(x-2)$ one finds $W[P](y)=-3y+\tfrac32y^{2}$, so the initial
velocities are $3/2$ at $\rho=1$ and $0$ at $\rho=2$: the larger zero is
momentarily stationary before turning back to meet its partner at $x_0=12/7$.
For $(x-1)(x-2)(x-3)$ the velocities are $23/12$, $-2/3$, and $3/4$.

\begin{remark}[The derivative at order zero]\label{rem:generator}
The operators $J^{\alpha}$ form a semigroup, $J^{\alpha}J^{\beta}=J^{\alpha+\beta}$,
so $n$ applications of the order-$1/n$ integral reproduce one ordinary
integration, and the family $r(\alpha)$ is induced by that semigroup rather than
by an interpolation chosen for convenience. The reduced family
$P\mapsto H(\alpha,\cdot)$ is not itself a semigroup, because the factor
$x^{\alpha}$ has been removed. Differentiating the coefficient
multiplier $g_j(\alpha)=\Gamma(j+1)/\Gamma(j+\alpha+1)$ of $H(\alpha,\cdot)$ at
$\alpha=0$ gives $g_j'(0)=-\psi(j+1)=\gamma-H_j$, so
\[
  \frac{\partial H}{\partial\alpha}\Big|_{\alpha=0}=\gamma P-W[P].
\]
For the semigroup itself, $\partial_\alpha J^{\alpha}P(x)|_{\alpha=0}
=(\log x+\gamma)P(x)-W[P](x)$ for $x>0$. At a zero $\rho$ of $P$ the
multiplication term vanishes, which is why the harmonic transform of
\eqref{eq:Wdef} governs the order-zero end.
\end{remark}

Differentiating once more gives $\frac{d^{2}}{d\alpha^{2}}\log g_j\big|_{0}
=-\psi'(j+1)=H_j^{(2)}-\zeta(2)$ with $H_j^{(2)}=\sum_{i\le j}i^{-2}$, and the
next term follows. Set
\begin{equation}\label{eq:W2def}
  W_2[P](y)=\sum_{j=0}^{m}a_j\bigl(H_j^{2}+H_j^{(2)}\bigr)y^{j}.
\end{equation}

\begin{proposition}[Second-order behavior at order zero]\label{prop:secondzero}
Let $\rho$ be a simple zero of $P$. Then
$r(\alpha)=\rho+r_1\alpha+r_2\alpha^{2}+O(\alpha^{3})$ with
$r_1=W[P](\rho)/P'(\rho)$ as in Proposition~\ref{prop:initvel} and
\begin{equation}\label{eq:r2}
  r_2=-\frac{\tfrac12W_2[P](\rho)-W[P]'(\rho)\,r_1+\tfrac12P''(\rho)\,r_1^{2}}
            {P'(\rho)} .
\end{equation}
Both $\gamma$ and $\zeta(2)$ cancel, the first as in
Proposition~\ref{prop:initvel} and the second because $P(\rho)=0$.
\end{proposition}

\begin{proof}
Exponentiating the expansion of $\log g_j$ gives
$g_j(\alpha)=1+\alpha(\gamma-H_j)
+\tfrac{\alpha^{2}}{2}\bigl[(\gamma-H_j)^{2}+H_j^{(2)}-\zeta(2)\bigr]+O(\alpha^{3})$,
so $H=P+\alpha(\gamma P-W)+\alpha^{2}K+O(\alpha^{3})$ where
$2K=\gamma^{2}P-2\gamma W+W_2-\zeta(2)P$. Substituting
$x=r(\alpha)$ into $H(\alpha,r(\alpha))=0$ and collecting powers of $\alpha$,
the linear term reproduces $r_1$, and the quadratic term is
$K(\rho)+(\gamma P-W)'(\rho)r_1+\tfrac12P''(\rho)r_1^{2}+P'(\rho)r_2=0$.
Since $P(\rho)=0$ the terms in $\gamma^{2}$ and $\zeta(2)$ vanish from
$K(\rho)$, leaving $\tfrac12(W_2(\rho)-2\gamma W(\rho))$, and the remaining
$\gamma$ cancels against $\gamma P'(\rho)r_1=\gamma W(\rho)$. Solving for $r_2$
gives \eqref{eq:r2}.
\end{proof}

For $P=(x-1)(x-2)$ this gives $r_2=1$ at $\rho=1$ and $r_2=-1$ at $\rho=2$; for
$(x-1)(x-2)(x-3)$, $r_2=761/288$ at $\rho=1$ and $15/32$ at $\rho=3$. All four
were confirmed against the implicit series. The two ends of $V$ are therefore
governed by formulas of the same shape with dual weights, as
Table~\ref{tab:ends} records.

\begin{table}[!ht]
\centering
\renewcommand{\arraystretch}{1.35}
\begin{tabular}{lll}
\hline
 & $\alpha=0$ & $\alpha\to\infty$\\
\hline
position fixed by & $\rho$, a zero of $P=\sum a_jy^{j}$
  & $\lambda$, a zero of $\LP=\sum a_j\,j!\,y^{j}$\\
weight sequence & harmonic, $H_j=\sum_{i\le j}1/i$
  & arithmetic, $\sigma_1(j)=\sum_{i\le j}i$\\
correction transform & $W[P]=\sum a_jH_jy^{j}$
  & $\tfrac12R_P=\sum a_j\,j!\,\sigma_1(j)y^{j}$\\
second-order weight & $H_j^{2}+H_j^{(2)}$, $H_j^{(2)}=\sum_{i\le j}i^{-2}$
  & $\sigma_1(j)^{2}+\sigma_2(j)$, $\sigma_2(j)=\sum_{i\le j}i^{2}$\\
result & velocity $W[P](\rho)/P'(\rho)$
  & offset $\CP(\lambda)=\tfrac12R_P(\lambda)/\LP'(\lambda)$\\
what cancels & $\gamma$, because $P(\rho)=0$
  & $\sigma_1(m)$, because $\LP(\lambda)=0$\\
\hline
\end{tabular}
\caption{The two ends of the trajectory curve $V$. At order $\alpha=0$ a simple
zero $\rho$ of $P$ starts with velocity $W[P](\rho)/P'(\rho)$
(Proposition~\ref{prop:initvel}); as $\alpha\to\infty$ the branch belonging to a
simple zero $\lambda$ of $\LP$ has offset $\CP(\lambda)$
(Theorem~\ref{thm:branch}). Here $H_j$ and $H_j^{(2)}$ are the harmonic numbers
of orders one and two, $\sigma_1(j)$ and $\sigma_2(j)$ are the sums of the first
$j$ positive integers and of their squares, $R_P$ is defined in \eqref{eq:RS},
and $\gamma$ is Euler's constant.}
\label{tab:ends}
\end{table}

In each case the answer is an auxiliary transform evaluated at the relevant
zero, divided by the derivative of the defining transform at that zero. In
each case, too, a constant independent of the zero ($\gamma$ or $\sigma_1(m)$)
drops out, because the zero annihilates the defining transform.

\section{Special-function illustrations}\label{sec:special}

\subsection{Harmonic numbers and \texorpdfstring{$e^{\gamma}$}{e\textasciicircum gamma}}\label{subsec:harmonic}
For $\varepsilon>0$ let $f_\varepsilon(x)=x^{1+\varepsilon}/(1+\varepsilon)-x$,
a binomial of the form of Section~\ref{sec:binomials} with $a=1+\varepsilon$,
$A=1/(1+\varepsilon)$, $b=1$, and $B=1$. Its positive zero after order-$k$
integration is
\[
  r_{k,\varepsilon}
  =\left[\frac{\Gamma(k+2+\varepsilon)}
              {\Gamma(1+\varepsilon)\Gamma(k+2)}\right]^{1/\varepsilon}.
\]
On taking logarithms and letting $\varepsilon\to0^{+}$, the difference
quotient of $\log\Gamma$ converges to its derivative, so
\[
  \lim_{\varepsilon\to0^{+}}\log r_{k,\varepsilon}=\psi(k+2)-\psi(1)=H_{k+1},
  \qquad\text{hence}\qquad
  \lim_{\varepsilon\to0^{+}}r_{k,\varepsilon}=e^{H_{k+1}}.
\]
Using $H_n=\log n+\gamma+\frac{1}{2n}+O(n^{-2})$ \cite[\S\S5.4, 5.11]{DLMF} we find
$e^{H_{k+1}}-e^{H_k}\to e^{\gamma}$. The reverse order of limits agrees, which
is a result rather than an assumption: for fixed $\varepsilon$,
Theorem~\ref{thm:binomial} gives the limiting spacing
\[
  \rho F(1+\varepsilon,1)
  =(1+\varepsilon)^{1/\varepsilon}
   \left[\frac{\Gamma(2)}{\Gamma(2+\varepsilon)}\right]^{1/\varepsilon}
  =\Gamma(1+\varepsilon)^{-1/\varepsilon},
\]
using $\Gamma(2+\varepsilon)=(1+\varepsilon)\Gamma(1+\varepsilon)$, and
$\log\Gamma(1+z)=-\gamma z+O(z^{2})$ \cite[\S5.7]{DLMF} gives
$\Gamma(1+\varepsilon)^{-1/\varepsilon}\to e^{\gamma}$. The scaled location $r_{k,\varepsilon}/k$ in fact converges
jointly in $k$ and $\varepsilon$. Since $r_{k,\varepsilon}=R_{k+1}(\varepsilon)$ with
$R_n(\varepsilon)=\prod_{j\le n}(1+\varepsilon/j)^{1/\varepsilon}$, the Weierstrass
product for $\Gamma$ \cite[\S5.8]{DLMF} gives
\[
  \log\frac{R_n(\varepsilon)}{n\,\Gamma(1+\varepsilon)^{-1/\varepsilon}}
  =\bigl(H_n-\log n-\gamma\bigr)
  +\sum_{j>n}\Bigl(\frac1j-\frac1\varepsilon\log\Bigl(1+\frac\varepsilon j\Bigr)\Bigr),
\]
where the first term lies in $(0,\frac1{2n})$ and the $j$th term of the sum in
$(0,\frac{\varepsilon}{2j^{2}})$. Hence
\[
  0<\log\frac{R_n(\varepsilon)}{n\,\Gamma(1+\varepsilon)^{-1/\varepsilon}}<\frac{1+\varepsilon}{2n}
  \qquad(n\ge1,\ \varepsilon>0),
\]
and $r_{k,\varepsilon}/k\to e^{\gamma}$ along every
joint path $k\to\infty$, $\varepsilon\to0^{+}$. More generally, for
fixed $a>b\ge0$ the same expansion yields
$\lim_{\varepsilon\to0^{+}}F(a\varepsilon,b\varepsilon)=e^{\gamma}$,
independently of the ratio $a:b$.

\subsection{Half-integer exponents and powers of
\texorpdfstring{$\pi$}{pi}}
For $f(x)=x-x^{1/2}$ we have $F(1,1/2)=[\Gamma(3/2)/\Gamma(2)]^{2}=\pi/4$. More
generally let $n$ and $m$ be nonnegative integers with $n>m+1/2$ and put
$d=n-m\ge1$, so that the exponent gap is $(2d-1)/2$ and
$F(n,m+1/2)=[\Gamma(m+3/2)/n!]^{2/(2d-1)}$. The Legendre duplication formula
\cite[\S5.5]{DLMF} gives
$\Gamma(m+3/2)=(2m+2)!\sqrt{\pi}/(4^{m+1}(m+1)!)$, whence
\begin{equation}\label{eq:pipower}
  F(n,m+1/2)=q^{2/(2d-1)}\,\pi^{1/(2d-1)},
  \qquad
  q=\frac{(2m+2)!}{4^{m+1}(m+1)!\,n!}\in\mathbb{Q}.
\end{equation}
For $d=1$ this is a rational multiple of $\pi$. For $d>1$ it cannot be. Indeed,
write $p=2d-1\ge3$ and suppose $q^{2/p}\pi^{1/p}=r\pi$ with $r$ rational and
positive. Raising to the $p$th power gives $q^{2}\pi=r^{p}\pi^{p}$, so
$\pi^{p-1}=q^{2}/r^{p}$ is rational with $p-1\ge2$, making $\pi$ algebraic and
contradicting the Lindemann transcendence theorem \cite{Lindemann1882}. The
first examples are
\[
  F(1,1/2)=\frac{\pi}{4},\quad
  F(2,3/2)=\frac{9\pi}{64},\quad
  F(2,1/2)=\Bigl(\frac{\pi}{16}\Bigr)^{1/3},\quad
  F(3,3/2)=\Bigl(\frac{\pi}{64}\Bigr)^{1/3}.
\]
Thus $\pi$ enters only through half-integer gamma values, that is, only through
Legendre duplication, and not through the general root-dynamics mechanism
itself. For integer exponents $F(n_1,n_2)$ is algebraic, and the map
$a_j\mapsto j!\,a_j$ introduces no $\pi$: if the coefficients of $P$ are
algebraic, so are the zeros of $\LP$.

\section{Multiple zeros: the Hermite law}\label{sec:multiple}

The simple-root hypothesis is essential for Theorem~\ref{thm:branch}. At a
multiple zero the implicit function theorem does not apply, and the scale is
set not by the first correction $\RP$ alone but by the whole coefficient
multiplier; by Remark~\ref{rem:compactC}, $\RP$ vanishes at every zero of
multiplicity at least three. The correct statement is the following.

\begin{proposition}[Multiple zeros: the Hermite law]\label{prop:multiple}
Let $\lambda\neq0$ be a zero of $\LP$ of multiplicity $q\ge2$, and let
$\xi_1,\dots,\xi_q$ be the zeros of the probabilists' Hermite polynomial
$\mathrm{He}_q$. For all large $k$, exactly $q$ zeros of $Q_{1/k}$ lie near
$\lambda$, and they satisfy $y=\lambda+\lambda\,\xi_i\,k^{-1/2}+O(k^{-1})$;
equivalently
\[
  r_k=k\lambda+\lambda\,\xi_i\sqrt k+O(1),\qquad i=1,\dots,q.
\]
Each of these zeros has an expansion in powers of $k^{-1/2}$, so
$r_k-r_{k-1}=\lambda+\lambda\xi_i/(2\sqrt k)+O(k^{-3/2})$. For odd $q$ the
central zero, with $\xi_i=0$, is holomorphic in $1/k$, and its spacing error is
$O(k^{-2})$. The same holds for
$J^{\alpha}P$ with $\alpha$ in place of $k$.
\end{proposition}

\begin{proof}
Put $h=1/k$ and $y=\lambda(1+\sqrt h\,w)$, so that $\theta=y\,d/dy$ acts on
polynomials in $w$ as $(h^{-1/2}+w)\,d/dw$. On polynomials of degree at most
$m$, $Q_h=c_\theta(h)\LP$ with $\log c_\theta(h)=-\frac h2\theta(\theta+1)
+\sum_{n\ge2}h^{n}p_n(\theta)$, where $p_n$ has degree $n+1$, by expanding
$-\sum_{\ell\le j}\log(1+\ell h)$ and Faulhaber's formula. In the variable
$w$ we have $h\,\theta(\theta+1)=d^{2}/dw^{2}+O(h^{1/2})$. For the rest, $\theta$
is diagonalizable on this space with eigenvalues $0,\dots,m$, so
$\rho_h(\theta)=\sum_{n\ge2}h^{n}p_n(\theta)$ equals its Newton interpolant
$\sum_{i=0}^{m}\delta_i(h)\prod_{l<i}(\theta-l)$, where $\delta_i(h)$ is the
$i$-th divided difference at $0,\dots,i$ of
$j\mapsto\sum_{n\ge2}h^{n}p_n(j)=\log c_j(h)+\frac h2j(j+1)$, holomorphic for
$|h|<1/m$. Since $\deg p_n=n+1$, $\delta_i(h)=O(h^{\max(2,i-1)})$, while
$\prod_{l<i}(\theta-l)=O(h^{-i/2})$ in the variable $w$; hence
$\rho_h(\theta)=O(h^{1/2})$, and
$c_\theta(h)=\exp(-\frac12d^{2}/dw^{2})+O(h^{1/2})$ on this finite-dimensional
space. Writing $\LP(y)=(y-\lambda)^{q}g(y)$ with $g(\lambda)\neq0$, we have
\[
  \LP\bigl(\lambda(1+\sqrt h\,w)\bigr)
  =(\lambda\sqrt h)^{q}g(\lambda)\bigl(w^{q}+O(h^{1/2})\bigr).
\]
Since $\exp(-\frac12d^{2}/dw^{2})w^{q}=\mathrm{He}_q(w)$,
\[
  (\lambda\sqrt h)^{-q}Q_h\bigl(\lambda(1+\sqrt h\,w)\bigr)
  =g(\lambda)\,\mathrm{He}_q(w)+O(h^{1/2})
\]
uniformly on compact sets. The zeros of $\mathrm{He}_q$ are real and simple, so
Rouch\'e's theorem gives exactly one zero $w_i=\xi_i+O(h^{1/2})$ near each
$\xi_i$. Finally, put $s=\sqrt h$. Before division, $Q_{s^{2}}(\lambda(1+sw))$ is a
polynomial in $w$ whose coefficients are holomorphic in $s$ near $0$, because
$Q_h$ involves only $h=s^2$ and the substitution only $s$. The estimate just
proved shows that each coefficient divided by $s^{q}$ stays bounded as
$s\to0$, so the singularity of the quotient at $s=0$ is removable and the
left-hand side is a polynomial in $w$ with coefficients holomorphic in $s$,
equal to $g(\lambda)\mathrm{He}_q(w)$ at $s=0$. By the implicit function theorem each $w_i$ is holomorphic in $s$. For odd
$q$ the equation is invariant under $(s,w)\mapsto(-s,-w)$, so uniqueness of
the zero near $w=0$ gives $w(-s)=-w(s)$; then $s\,w(s)$ is even in $s$ and
that zero is holomorphic in $h$. Hence
$r_k=k\lambda(1+k^{-1/2}w_i(k^{-1/2}))$ expands in powers of $k^{-1/2}$, and
differencing gives the spacing law.
\end{proof}

For $P(x)=x^{2}-2\sqrt2\,x+1$, whose transform has the double zero
$\lambda=1/\sqrt2$, the quotient $(r_k-k\lambda)/\sqrt k$ takes the values
$-0.5727$, $-0.6930$, $-0.7057$, and $-0.70697$ at $k=10^{2},10^{4},10^{6}$, and
$10^{8}$, approaching $\lambda\xi_1=-1/\sqrt2=-0.70711$. For the triple zero of
$\LP=(y-1)^{3}(y+2)$ the three quotients at $k=10^{6}$ are $-1.7284$, $0.0027$, and $1.7357$, against $-\sqrt3$, $0$, and $\sqrt3$. The limiting spacing survives,
$r_k-r_{k-1}=\lambda+\lambda\xi_i/(2\sqrt k)+O(k^{-3/2})$, but with the slower
rate $k^{-1/2}$ in place of the $O(k^{-2})$ of Corollary~\ref{cor:discrete}. The earlier version of this paper predicted a scale $k^{1-1/q}$ under the condition $\RP(\lambda)\neq0$; that
condition cannot hold for $q\ge3$, and Proposition~\ref{prop:multiple} replaces
the prediction. Zeros that are distinct but at mutual distance of order $|\lambda|k^{-1/2}$ lead in the same way to the polynomial
$\exp(-\frac12d^{2}/dw^{2})\prod_i(w-\beta_i)$ in place of $\mathrm{He}_q$; we do
not pursue that cluster case here. The identity
$\exp(-\frac t2\,d^{2}/dw^{2})\,w^{q}=t^{q/2}\mathrm{He}_q(w/\sqrt t)$ behind
the proof is the classical heat-polynomial identity \cite{RosenbloomWidder1959}.
For the heat flow applied to high powers of a polynomial, the corresponding
spreading of each zero is studied by H\"ofert, Jalowy, and Kabluchko
\cite{HoefertJalowyKabluchko2026}.

\section{Limitations and open problems}\label{sec:limitations}

The transformed zero $\lambda=0$ needs no separate treatment, by
Proposition~\ref{prop:pinned}. The degeneration of Section~\ref{subsec:degeneration} is regular in its
simplest form. If $a_0$ varies with $a_1\neq0$ and the other coefficients
fixed, then $Q_h(y)=a_0+a_1y/(1+h)+O(y^{2})$ near $(a_0,h,y)=(0,0,0)$, and
the implicit function theorem in $(a_0,h)$ jointly gives
$y(a_0,h)=-(1+h)a_0/a_1+O(a_0^{2})$ uniformly for small $h$. Hence
$r(\alpha)=-(\alpha+1)a_0/a_1+O(\alpha a_0^{2})$, and with $a_0=\tau/\alpha$
the zero tends to $-\tau/a_1$; this limit explains the product scaling. What
remains open are the degenerate cases: $a_1$ vanishing as well, higher
multiplicity at the origin, or interaction with other nearby zeros, together
with their relation to the merging of discriminant values at $\alpha=-2$ in
Proposition~\ref{prop:degeneration}.

The results are not uniform in the degree. Lemma~\ref{lem:coeff} guarantees the
Taylor representation about $h=0$ only on the common disc $|h|<1/m$,
corresponding to integer samples $k>m$; that analytic radius shrinks with the
degree. This nonuniformity is not merely an
artifact of the method: $\CP$ can itself grow with $m$, as the following family
illustrates. For
$P=\prod_{i=1}^{m}(x-i)$ the largest value of $|\CP(\lambda)|/m$ over the zeros of $\LP$ is
$0.80,0.93,0.92,0.93,0.95$, and $0.96$ at $m=2,3,6,8,12$, and $16$. For the tested
degrees the largest constant term is approximately proportional to $m$,
suggesting an absolute correction of size $m/k$ in the scaled roots. These
computations do not establish an asymptotic bound as $m\to\infty$ or a necessary
degree-uniform condition on $k$, and we do not claim a bound uniform in the degree.
Numerically, $|r_k/k-\lambda|$ is close to $m/k$ across that range; this agreement is an
empirical observation, and we make no optimality claim for $m/k$.

A second scale is visible in the coefficient family itself. Summing the
logarithms in \eqref{eq:cjh} gives
\begin{equation}\label{eq:gaussian}
  \prod_{\ell=1}^{j}\Bigl(1+\frac{\ell}{k}\Bigr)^{-1}
  =\exp\Bigl[-\frac{j(j+1)}{2k}+O\Bigl(\frac{j^{3}}{k^{2}}\Bigr)\Bigr],
\end{equation}
uniformly for $j\le ck$ with fixed $c<1$. Hence $c_j(1/k)$ departs appreciably
from $1$ once $j$ is of order $\sqrt k$. At
$k=10^{4}$ the ratio of $c_j(1/k)$ to $\exp[-j(j+1)/(2k)]$ is $1.0017$ at
$j=\sqrt k$ and $1.013$ at $j=2\sqrt k$. Whether that is also the threshold for
the root expansion depends on the polynomial, since the expansion degrades when
$\DP/\CP$ becomes comparable with $k$, and that ratio is not of a fixed order in
$m$. Empirically, the largest value of $|\DP/\CP|$ over the zeros of $\LP$ grew
roughly like $m$ for $P=\prod_{i=1}^{m}(x-i)$, from $0.27$ at $m=2$ to $4.0$ at
$m=20$, and also for $a_j\equiv1$. For the truncated exponential $a_j=1/j!$
(Example~1 of \cite[\S6]{BorweinChenDilcher1995}) it grew like $m^{2}$: the
successive ratios at $m=4,8,16,32,64$ approach $4$.

This fixed-degree analysis leaves a double-scaling problem untouched: the behavior of the zeros
when the degree grows with the integration order, $m=m(k)$, with the regime
$m\asymp\sqrt k$ singled out by \eqref{eq:gaussian}. When $m\asymp k$ and the
zeros of $\LP$ are nonnegative, Remark~\ref{rem:ffree} gives the limiting zero
distribution: free multiplicative convolution with a Marchenko--Pastur law.
The zero
asymptotics for $m\asymp\sqrt k$, and the conditions under which the fixed-degree
limit extends to analytic inputs, remain natural open problems; a degree-uniform statement would
require a different normalization, and the present finite-degree argument does
not pass directly to an infinite series.

Beyond the degree, the admissible $h_0$ for a selected zero depends also on the
separation of $\lambda$ from the other zeros of $\LP$.
Explicit thresholds can be obtained from a contour lower bound for $|\LP|$
together with coefficient bounds; no uniform sharp threshold is asserted here,
and the localization used here is the annulus
$c_-(P)\,k<|z|<c_+(P)\,k$ of \cite[\S6]{BorweinChenDilcher1995}.

By Proposition~\ref{prop:endpointall}, in every degree the nonzero zeros of
a real polynomial stay real at every positive order exactly when $\LP$ is
real-rooted, and they then stay simple. So if $\LP$ is real-rooted, no collision
of nonzero zeros occurs at positive order, and $\Phi_P$ has no positive zero
when $a_0\neq0$. Without this hypothesis, a generic collision satisfying the
nondegeneracy conditions of Remark~\ref{rem:puiseux} has the square-root local
form; nongeneric collisions need not cause the zeros to leave the real axis. This extends
Corollary~\ref{cor:endpoint}.

In higher degree two questions about the collision polynomial $\Phi_P$ of
\eqref{eq:spurious} remain open: whether the positive zeros of $\Phi_P$ are generically
simple, so that every collision is an ordinary square-root branch point as in
degree two, and whether the bound $m(m-1)/2$ is attained.

Two directions suggested by neighboring literature deserve explicit mention.
First, the results here are fixed-degree. For repeated \emph{differentiation}
the large-degree behavior of the root density is described by a nonlocal
transport equation \cite{Steinerberger2019}, rigorously justified in the
periodic setting for a class of trigonometric polynomials \cite{KiselevTan2022},
and is formally identified with a fractional free convolution power
\cite{Steinerberger2023,ShlyakhtenkoTao2020}. Hall, Ho, Jalowy, and Kabluchko also analyze large-degree repeated
integration \cite[\S8]{HallHoJalowyKabluchko2026} through the case $b<0$ of
$z^{a}(d/dz)^{b}$. They treat Kac, Weyl, and exponential random polynomials, for
which the limiting zero distribution can concentrate mass on circles. Those results
concern the zero distribution as the degree grows, whereas the results here
fix the polynomial and follow individual zeros. Related
zero-dynamics questions arise for Lee--Yang zeros: Kabluchko
\cite{Kabluchko2025} relates the unitary Hermite polynomials and their zeros to
the backward heat flow, and identifies the limiting Lee--Yang zero distribution
of the Curie--Weiss model.

Second, by Remark~\ref{rem:expectation} the transform is
$\LP(x)=\mathbb{E}[P(xT)]$ for $T$ standard exponential, so the factorial
weights are simply the moments of $T$. Replacing $T$ by another positive random
variable gives $L_T[P](x)=\sum_j a_j\,\mathbb{E}[T^{j}]\,x^{j}$, and one may ask
for which moment sequences an analogue of the scaling theorem and of
Theorem~\ref{thm:branch} holds, and which operators such weights arise from.

The monodromy of the covering $\pi$ is the full symmetric group under the
nonvanishing conditions stated after Proposition~\ref{prop:curve}. Classifying
the exceptional coefficient configurations not covered by those sufficient
conditions is open.

Remark~\ref{rem:genus} gives only the generic Newton-polygon bound on the
genus of the trajectory curve. Two concrete questions are left open here: computing that genus exactly,
which requires the singularity structure of $V$; and identifying the curve
(elliptic in the nondegenerate case) that arises for a given cubic $P$.

\section{Scope and novelty}\label{sec:scope}

The main ingredients are classical: repeated integration with zero constants;
gamma-ratio asymptotics \cite{DLMF}; Riemann--Liouville fractional
integration
\cite{Podlubny1999,SamkoKilbasMarichev1993,KilbasSrivastavaTrujillo2006};
Rouch\'e's theorem, the argument principle, and the residue calculus
\cite{Ahlfors1979,Marden1966}; the inverse Borel coefficient
transform \cite{Hardy1949}; the classical theory of multiplier and
zero-diminishing transformations
\cite{PolyaSchur1914,RahmanSchmeisser2002,CravenCsordas1977,CravenCsordas1995,CravenCsordas2004,PratherShaw1983}; Stieltjes's discriminant of the Laguerre
polynomials \cite{Szego1975}; and the heat polynomials
\cite{RosenbloomWidder1959}. The leading scaling
$r_k/k\to\lambda$ is due to Borwein, Chen, and Dilcher
\cite{BorweinChenDilcher1995} and is not claimed as new. In
Section~\ref{sec:structural} the zero-diminishing property, the P\'olya--Schur
criterion, and Szeg\H o's theorem on truncated exponentials are quoted rather
than proved; Propositions~\ref{prop:noncreation}
and~\ref{prop:hermite} are proved here.

The comparisons drawn above were developed with reference to
\cite{BorweinChenDilcher1995,BelloPijeiraRivero2020,CaparattaPauliSaidak2025,WaniSheikh2026,HallHoJalowyKabluchko2026},
which the author has reviewed and takes responsibility for; the remaining
references are cited for the specific results attributed to them. The attribution of the leading scaling to
\cite[Prop.~6.4]{BorweinChenDilcher1995} is exact, including the normalization
and the notation $L$. Within the literature reviewed we did not locate the fixed-polynomial
large-order formulas proved here for $\CP$ and $\DP$, the resulting
discrete-spacing and continuous-velocity pairing, or the identification of
the asymptotic line. The quadratic collision value is that of
\cite[Cor.~3.8]{CaparattaPauliSaidak2025} after reversing the order
(Remark~\ref{rem:cps}). Later work on Borwein--Chen--Dilcher polynomials
\cite{BelloPijeiraRivero2020}, on zero flows under differentiation and
integration \cite{MartinezFinkelshteinRakhmanov2024,HallHoJalowyKabluchko2026},
and on the heat flow \cite{HoefertJalowyKabluchko2026} concerns families
whose degree grows. This comparison is a delimited literature statement, not a
definitive priority claim. The plane-curve reading of
Section~\ref{sec:curve} is offered as a synthesis rather than as new
mathematics: the curve is classical in type, and
Propositions~\ref{prop:curve} and~\ref{prop:asymptote} are short consequences
of its Newton polygon. Its value is that it supplies a second and independent
derivation of $\CP$ and a structural explanation for the features that
distinguish degree two. Propositions~\ref{prop:pinned}
and~\ref{prop:degeneration} are likewise elementary. They are included because
they close a scope gap. The zero $\lambda=0$, which a statement restricted to
nonzero zeros leaves as an exceptional case, is in fact rigid. The
degeneration in which a nonzero transformed zero is driven to the origin is a
regular perturbation in its simplest case and a specific unfolding question in
the degenerate ones.

The higher-degree factorization \eqref{eq:spurious} is proved by a
generic-coefficient valuation argument followed by polynomial-identity
continuation. The argument uses the Eisenstein property of the degree-$i$
Hensel factor at $\alpha=-i$ and is logically independent of computation. The
low-degree symbolic checks recorded below are corroboration and not part of
any proof.
The constants and asymptotic rates listed here were verified in exact rational arithmetic and in high-precision arithmetic
carried to forty significant digits. This verification establishes agreement to that
working precision, not certified bounds; the scripts are available from the
author. The items checked are the expansions of
Theorem~\ref{thm:branch} on eleven distinct real
and complex branches, the exact quadratic of Section~\ref{subsec:quadexact}, the
constants of Section~\ref{subsec:complex}, the exact polygon identity
\eqref{eq:polygon} at several noninteger orders, the collision data of Theorem~\ref{thm:collision}, and identities
\eqref{eq:endpoints} and \eqref{eq:spurious} for $m\le5$. The results added since the earlier version were checked in the same way:
the operator form of Remark~\ref{rem:torus}, the three representations of
Remark~\ref{rem:ffree} and its Laguerre asymptotics, the sum rules
\eqref{eq:sumrules}, the numerical examples of
Proposition~\ref{prop:multiple} and its spacing law, the order-zero values of
Propositions~\ref{prop:initvel} and~\ref{prop:secondzero}, and the real-zero
counts of Proposition~\ref{prop:hermite} for $n\le12$.

\section{Conclusion}\label{sec:conclusion}

For a fixed polynomial the factorially weighted transform
$\LP(z)=\sum_{j=0}^{m}a_j\,j!\,z^{j}$ is the classical root-scale limit under
repeated integration with zero constants. Retaining two further terms in the
scaled coefficient family gives, at a simple nonzero zero $\lambda$ of $\LP$,
$r_k=k\lambda+\CP(\lambda)+\DP(\lambda)/k+O(k^{-2})$: the constant term
supplies a parallel asymptotic line, while the $1/k$ term controls the leading
correction to successive spacing. The same local branch, under $h=1/\alpha$,
gives the continuous Riemann--Liouville trajectory and its velocity. Because
the two scaled families are literally the same rational family, the discrete
sequence is a sampling of the continuous trajectory. The discrete spacing is
the exact mean of the continuous velocity over one unit of order. Hence
$\lambda$ carries three equivalent asymptotic meanings: scaled root location,
discrete spacing, and continuous-order velocity.

The binomial case makes the mechanism explicit through the Factorial Fraction
and yields an exact regular-polygon configuration at every positive order,
while the quadratic case shows how real branches collide and become complex at
a finite fractional order determined by the two discriminants. In every degree the collision orders are the positive zeros of a
discriminant that carries a universal factor at negative orders, and the
real-rootedness of $\LP$ alone decides whether the zeros stay real, and
simple, at every positive order. At a multiple zero of $\LP$ the expansion in $1/k$ gives way to
a Hermite law on the scale $\sqrt k$, and the whole family of trajectories is a
single plane curve whose two ends are governed by dual transforms. Together these
results sharpen the classical scaling theorem without changing its underlying
transform, and they isolate concrete directions for a cluster and
higher-degree theory.

\section*{Origin of the results}
The problem studied here was formulated by the author in an earlier unpublished manuscript, \emph{The Factorial Fraction: Special Limits Emerging from the Difference Between the Positive Root of Certain Polynomials and the Positive Root of their Derivatives as Exponentiation Approaches Infinity}, on the Factorial Fraction $F(n_1,n_2)=(n_2!/n_1!)^{1/(n_1-n_2)}$.
That manuscript introduced $F$ as the organizing quantity. It posed the
problem of following the zeros of repeated integrals taken with every constant
of integration equal to zero, and it worked out the quadratic and binomial
cases. It also found that $e^{\gamma}$ and $\pi/4$ arise from $F$, as a limit at small
exponent and at a half-integer argument respectively. The present
paper develops that program. The author also asked whether half-integrals, third-integrals, and integrals of vanishing order fit the theory; this question led to
the Riemann--Liouville embedding of Section~\ref{sec:parameter}. The author's
questions on the Factorial Fraction, on number theory, and on the geometry of
the zeros opened the conversations from which three further contributions came:
the identification of the regime $m\asymp\sqrt k$ in Section~\ref{sec:limitations}, the monodromy
paragraph after Proposition~\ref{prop:curve}, and the swept-area law of
Remark~\ref{rem:superconv}. The operator and the leading scaling limit
were found earlier by Borwein, Chen, and Dilcher \cite{BorweinChenDilcher1995},
and the evaluations behind the two special values are classical. The first version of this paper is arXiv:2610.12498v1 \cite{LawlessARDv1}. The
present version adds
Remarks~\ref{rem:torus}--\ref{rem:superconv}, \ref{rem:nochaos},
and~\ref{rem:area}; Section~\ref{subsec:sumrules};
Proposition~\ref{prop:endpointall}; Remark~\ref{rem:negorder}; the paragraph on
monodromy after Proposition~\ref{prop:curve}; Remark~\ref{rem:cps}; the uniform bound in Section~\ref{subsec:harmonic}; and
Section~\ref{sec:multiple} with Proposition~\ref{prop:multiple}, which replaces
a heuristic of the earlier version. It also adds a short
paragraph on motivating settings in the introduction and the comparisons with
Laguerre discriminants, Laguerre Galois groups, and heat polynomials.

\section*{Tool and computational resource disclosure}
The author used large language models from two developers, in distinct roles.
Claude (Anthropic; Claude Opus~5 and Claude Opus~5.5) was used for literature
identification and verification, symbolic and numerical checking, drafting, and typesetting. Several results took their present form in that work, among them
the expansion of Theorem~\ref{thm:branch} and its consequences, the
identification in Section~\ref{subsec:identification}, the collision results of
Section~\ref{sec:collision} other than those credited below, the plane-curve treatment of
Section~\ref{sec:curve}, and Proposition~\ref{prop:multiple}. Models from
OpenAI (GPT-5.6 Sol, GPT-6.0 Astra, and GPT-6.1 Sol) were used in exploratory conversations. These models proposed, among other things, the connection of the binomial
Factorial Fraction to the general transform $\LP$, the Riemann--Liouville
embedding, the initial-value-problem formulation of Section~\ref{sec:classical},
the exact binomial configuration \eqref{eq:polygon}, the general discriminant
criterion of Proposition~\ref{prop:discpoly} and Corollary~\ref{cor:gencoll},
the square-root model of Remark~\ref{rem:puiseux}, the Laplace representation \eqref{eq:laplace},
Proposition~\ref{prop:pinned}, the representations of Remark~\ref{rem:ffree},
the endpoint criterion of Proposition~\ref{prop:endpointall}, the operator form
of Remark~\ref{rem:torus}, the superconvergence index of
Remark~\ref{rem:superconv}, the uniform bound in Section~\ref{subsec:harmonic}, and the full-monodromy observation after Proposition~\ref{prop:curve}; they also
corrected earlier drafts, among other points in the radius of holomorphy in
Lemma~\ref{lem:coeff} and the treatment of the direction $[0:1]$ in
Proposition~\ref{prop:curve}.
GPT-6.0 Astra also reviewed late drafts; its findings, including the prior
quadratic collision result credited in Remark~\ref{rem:cps}, the centroid
normalization in Remark~\ref{rem:ffree}, the simplicity statement of
Proposition~\ref{prop:endpointall}, a counterexample correcting a collision
statement in Section~\ref{sec:limitations}, and the boundary-convergence
qualification in Remark~\ref{rem:negorder}, were checked and incorporated. Numerical verification used Python: exact rational arithmetic, high-precision
decimal arithmetic, and NumPy and SciPy for floating-point work, as described
in Section~\ref{sec:scope}. No language model is an author. These uses are reported in accordance with
arXiv's policy on authors' use of generative AI language tools. The
author directed the work; reviewed the definitions, statements, and proofs; and
takes responsibility for the content, including any errors.

\section*{Funding}
This research received no external funding.

\section*{Conflict of interest statement}
The author declares no conflicts of interest.

\section*{Data availability statement}
This study used no external or experimental datasets. The scripts and
computational outputs supporting the exact and numerical checks reported in
the paper are available from the author on reasonable request.

\providecommand{\bysame}{\leavevmode\hbox to3em{\hrulefill}\thinspace}
\providecommand{\MR}{\relax\ifhmode\unskip\space\fi MR }
\providecommand{\MRhref}[2]{%
  \href{http://www.ams.org/mathscinet-getitem?mr=#1}{#2}
}
\providecommand{\href}[2]{#2}

\end{document}